\documentclass{amsart}
\usepackage{amsmath, amscd, amssymb, latexsym, hhline, mathrsfs, bbm}
\usepackage[all]{xy}
\usepackage{enumitem}
\usepackage{soul}
\usepackage[usenames]{color} 

\def\1{{\mathbbm{1}}}
\def\a{{\alpha}}
\def\b{{\beta}}
\def\g{{\gamma}}

\def\w{{\omega}}

\def\l{{\lambda}}
\def\s{{\sigma}}
\def\t{{\tau}}

\def\G{{\Gamma}}
\def\ee{{\mathbf{e}}}

\def\BQ{{\mathbb{Q}}}
\def\BC{{\mathbb{C}}}
\def\BR{{\mathbb{R}}}

\def\BZ{{\mathbb{Z}}}
\def\bb{\mathbf{b}}

\def\AA{\mathscr{A}}
\def\CC{\mathscr{C}}

\def\BB{\mathscr{B}}

\def\diag{\operatorname{diag}}

\newcommand\irr{\operatorname{irr}}
\newcommand\SC{\operatorname{SC}}
\newcommand\Rep{\operatorname{Rep}}
\newcommand\Irr{\operatorname{Irr}}

\newcommand\ord{\operatorname{ord}}

\newcommand\Aut{\operatorname{Aut}}
\newcommand\id{\operatorname{id}}
\newcommand\SVec{\operatorname{sVec}}
\newcommand\FPdim{\operatorname{FPdim}}
\newcommand\FSexp{\operatorname{FPexp}}

\DeclareMathOperator\res{\operatorname{res}}

\def\to{\rightarrow}
\def\dim{{\mbox{\rm dim}}}
\def\ord{{\mbox{\rm ord}}}
\def\Hom{{\mbox{\rm Hom}}}

\def\Mod{{\text{-mod}}}
\def\R{{\Rep}}

\newcommand\inv{^{-1}}

\newcommand\ol[1]{\overline{#1}}

\renewcommand\o{\otimes}

\newtheorem{thm}{Theorem}[section]

\newtheorem{cor}[thm]{Corollary}
\newtheorem{prop}[thm]{Proposition}
\newtheorem{lem}[thm]{Lemma}

\newtheorem{q}[thm]{Conjecture}
\theoremstyle{definition}

\newtheorem{example}[thm]{Example}

\theoremstyle{remark}
\newtheorem{remark}[thm]{Remark}

\makeatletter
\def\namelabel#1#2{\@bsphack
 \protected@write\@auxout{}%
 {\string\newlabel{#1.nme}{{#2}{#2}}}%
 \@esphack}

\makeatother

\begin{document}
\title{Modular quasi-Hopf algebras and groups with one involution}

\author{Geoffrey Mason}
\address{Department of Mathematics, UC Santa Cruz, Ca 95064}
\email{gem@ucsc.edu}
\thanks{The first author is supported by Simons grant $\#62524$}
\author{Siu-Hung Ng}
\address{Department of Mathematics, Louisiana State University, Baton Rouge, LA 70803}
 \email{rng@math.lsu.edu}
 \thanks{The second author was partially supported by NSF grant DMS1664418}
 \maketitle
 
 \begin{abstract} In a previous paper the authors constructed a class of quasi-Hopf algebras $D^{\w}(G, A)$ associated to a finite group $G$, generalizing the twisted quantum double construction.\ We gave necessary and sufficient conditions, cohomological in nature, that the corresponding module category
 $\Rep(D^\omega(G, A))$ is a modular tensor category.\ In the present paper we verify the cohomological conditions for the class of groups $G$ which \emph{contain a unique involution}, and in this way we obtain an explicit construction of a new class of modular quasi-Hopf algebras.\ We develop the basic theory for general finite groups $G$, and also a parallel theory concerned with the question of when $\Rep(D^\omega(G, A))$ is super-modular rather than modular. We give some explicit examples involving binary polyhedral groups and some sporadic simple groups.
\end{abstract}

\medskip

\ Table of Contents\\
1.\ Introduction and statement of results.\\
2.\ Background and preliminary results.\\
3.\ $\omega$-admissibility.\\
3.1\  Admissible cocycles\\
3.2\ Admissibility and groups with one involution.\\
3.3\ Admissibility and representation groups.\\
3.4\ Examples.\\
4.\  Modularity and Super-modularity.\\
4.1\  Modular and Super-modular tensor categories.\\
4.2\  Modularity of $D^{\omega}(G, A)$.\\
4.3\  Proofs of the main Theorems.\\
5.\  Realizations of some Modular tensor Categories \\
5.1\  Reconstruction. \\
5.2\ A Conjecture for Binary Polyhedral Groups.\\
5.3\ The case of type $A$.\\
5.4\ The modular data of $V^{A_4}$ and $V^{D_4}$\\
5.5\ Two sporadic examples involving $2J_2$ and $Co_0$.

\section{Introduction and statement of results}
Suppose that $G$ is a finite group and $\w \in Z^3(G, \BC^{\times})$ a normalized, multiplicative
$3$-cocycle.\ The
\emph{twisted quantum double} $D^{\w}(G)$, widely studied since its introduction in \cite{DPR}, is a
quasi-Hopf algebra canonically attached to this data.\ A fundamental property of this class of
quasi-Hopf algebras is that they are \emph{modular} in the sense that
the module category $\Rep(D^{\w}(G))$ is a \emph{modular tensor category}.\ This follows from some
remarkable results of M\"{u}ger \cite[Theorem 3.16 and Proposition 5.10]{Mu}.

\medskip
In \cite{MN2} we introduced a \emph{generalization} of the twisted quantum double construction,
denoted by
$D^{\w}(G, A)$.\ The new ingredient is a central subgroup $A \subseteq Z(G)$, the case $A = 1$ being
the original twisted quantum double of $G$.\ By its very definition, $D^{\w}(G, A) = \BC^G_{\w}\#_c\BC(G/A)$ is a
cleft extension, where the subscript $c$ denotes some cohomological data
associated to $\omega$ and satisfying compatibilities sufficient to ensure that $D^{\w}(G, A)$ is a
quasi-Hopf algebra. (Further details
about this and other cohomological technicalities will be enlarged upon in Section 2.)\ The two main
results of
\cite{MN2} are essentially as follows (a precise formulation is given below):

\medskip\noindent
(i) necessary and sufficient conditions that there is a surjective morphism of quasi-bialgebras
(indeed, of quasi-Hopf algebras)
$\varphi : D^{\w}(G){\rightarrow} D^{\w}(G, A)$ which preserves the associated cleft extensions.

\medskip\noindent
(ii) assuming that $\varphi$ exists, necessary and sufficient conditions that $\Rep(D^{\w}(G, A))$ is
a modular tensor category.

\medskip
In both (i) and (ii), the necessary and sufficient conditions are cohomological in nature and it is
usually nontrivial to decide when they are satisfied by a given triple $(G,A, \w)$.\ The main purpose of the present paper is to
present an infinite class of groups
for which the cohomological conditions are indeed satisfied.\ What obtains is an infinite class of
modular quasi-Hopf algebras, almost all of which were unknown before now.

\medskip
It transpires that the case when $|A| = 2$ is particularly interesting, and it is this case that
mainly concerns us here.\ Indeed,
we will consider something stronger, namely finite groups $G$ which have a \emph{unique} subgroup
$A$ of order $2$.\ (The containment $A \subseteq Z(G)$ is
an immediate consequence.)\ This is a famous class of groups:\ the Sylow
$2$-subgroups of
$G$ are either cyclic or generalized quaternion, and the Brauer-Suzuki theorem
(\cite{G}, Chapter 12)
\emph{classifies} the possible quotient groups $G/O(G)$.\ ($O(G)$ is the largest normal subgroup of
$G$ of odd order.)\
Cohomologically, the Artin-Tate theory \cite{CE} says that this
class of groups has \emph{$2$-periodic cohomology}.

\medskip
 The Artin-Tate theory will be indispensable for the proof of the our main Theorem.\ The result of
 Brauer-Suzuki
points to interesting families of generalized twisted quantum doubles $D^{\w}(G, A)$ which are
modular by our results.\
Among these, we mention a family with $G = SL_2(q)$\ ($q$ any odd prime power) and a family 
for which
$G$ is a binary polyhedral group.\ See Subsection \ref{SSex} for further background and additional examples.\ Note that for
any choice of $G$, there
will generally be \emph{many} choices of $\w$ for which modularity holds.\ We make this precise in the
statement of the main
Theorem, to which we now turn.

\medskip
In order to state our main Theorem, we need a first installment of the results of Artin-Tate,
namely that
if $G$ has a unique subgroup $A$ of order $2$ then \emph{$4$ is a $2$-period for $G$}.\ Thus
the $2$-torsion subgroup $H^4(G, \BZ)_2$ of the fourth cohomology $H^4(G, \BZ)$ is a cyclic
group of order equal to the
$2$-part $|G|_2$ of $|G|$ (i.e., the order of a Sylow $2$-subgroup $G_2$ of $G$).\ We call any generator
of $H^4(G, \BZ)_2$ a
\emph{$2$-generator}, and we say that a cohomology class $\a \in H^4(G, \BZ)$ \emph{contains a
$2$-generator} if
the subgroup $\langle \a \rangle$ contains a $2$-generator.

\medskip
In the applications to quasi-Hopf algebras, we usually use \emph{multiplicative} cocycles.\ Thanks to
the isomorphisms
$H^n(G, \BZ) \cong H^{n-1}(G, \BC^{\times})$, it is easy to pass back and forth between additive and
multiplicative coefficients.\ We will usually not comment on this, except to say that the Artin-Tate
theory can, and will, be stated and used with $\BC^{\times}$-coefficients.
We can now state what is perhaps our main result.

\medskip\noindent
$\bf{Main\ Theorem}$.\ Suppose that $G$ is a finite group that contains a \emph{unique} subgroup
$A$ of order $2$, and let $\omega \in Z^3(G, \BC^{\times})$ be a normalized $3$-cocycle with $[\omega]$
the corresponding class in $H^3(G, \BC^{\times})$.\ Then the following are equivalent:
\begin{eqnarray*}
&&(a)\ D^{\w}(G, A)\ \mbox{is a modular quasi-Hopf algebra},\\
&&(b)\ \mbox{$[\omega]$ contains a $2$-generator}.
\end{eqnarray*}
 Exactly one half of the classes in $H^3(G, \BC^{\times})$ satisfy (b).

\medskip
The first three Sections of the paper are devoted to developing 
some basic facts about $D^{\omega}(G, A)$ and its module category.\ In the fourth Section we apply these results to prove the Main Theorem.\ We also develop two separate, but related, contexts: (i) analogs of the Main Theorem for some groups $G$ having a center of order 2 but more than one involution. We illustrate with two particularly interesting examples in which $G$ is the Schur cover $2.Co_1=Co_0$ or $2.J_2$ of one of the sporadic simple groups $Co_1$ or $J_2$ respectively; (ii) criteria for recognizing when $\Rep(D^{\omega}(G, A))$ is not a modular tensor category but rather a \textit{super-modular} tensor category.\ As an example, and in contrast to the Main Theorem, we prove (Theorem \ref{thmspor}) that, for $G=Co_0$,
$\Rep(D^{\omega}(G, A))$ is a super-modular tensor category if, and only if, $[\omega]$ contains a $2$-generator.\ 

\medskip
In addition to these two sporadic examples, the final Section of the paper is concerned with the \textit{reconstruction problem for modular tensor categories}.\ Given a MTC
$\mathcal{C}$, this asks:\ can we find a strongly regular vertex operator algebra $V$ (informally, a well-behaved VOA) for which there is an equivalence
of modular tensor categories $V$\Mod\ $\simeq \mathcal{C}$?\ When $G$ is a cyclic group of even order, we  show that 
$\Rep(D^\w(G, A)) \simeq V^{\ol G}_{A_1}\Mod$  where $A_1$ is the root lattice of type $A_1$, $V_{A_1}$ is the associated lattice theory VOA,
$A\subseteq G$ has order $2$,
and $\ol G = G/A$.\ We further give strong evidence, in terms of modular data, for similar equivalences for other binary polyhedral groups.\ 
Inspired by these empirical facts, we present two precise conjectures concerning an equivalence of $\Rep(D^\w(G,A))$ and the module categories of the $\ol G$-orbifold of the lattice VOAs $V_{A_1}$ and $V_{E_7}$. These conjectures provide explicit affirmative answers to the reconstruction problem when $G$
is a binary polyhedral group or $2.J_2$ respectively.\ We refer the reader to Subsection \ref{SSintro} for further details.

\section{Background and preliminary results}
With the exception of the coefficients of cohomology, all groups considered will be finite.\ We
 use standard notation, in particular $Z(G)$ is the \emph{center} of $G$, $G'$ the
\emph{commutator subgroup},
$\widehat{G} := H^1(G, \BC^{\times})$ is the group of \emph{characters}, and if $p$ is a prime then
$G_p$ is a Sylow $p$-subgroup of $G$.\ For $g \in G$, the \emph{centralizer} of $g$ in $G$ is $C(g) := \{x\in G \mid gx = xg\}$.\

\medskip
Let $\w$ be a normalized 3-cocycle of $G$. For $g, x, y \in G$, we define
\begin{eqnarray}
&&\theta_g(x,y) := \frac{\w(g, x, y)\w(x, y, g^{xy})}{\w(x, g^x, y)}\label{eq:01},\\
&&\gamma_g(x,y) := \frac{\w(x, y, g)\w(g, x^g, y^g)}{\w(x, g, y^g)}\label{eq:02}\,.
\end{eqnarray}
We have
\begin{eqnarray}
&&\ \ \ \ \ \ \ \ \ \ \ \ \ \ \ \ \ \ \ \ \ \ \theta_g(x, y)\theta_g(xy, z) = \theta_{g^x}(y, z)\theta_g(x, yz),\ \label{thetaid}\\
&&\theta_g(x, y)\theta_h(x, y)\gamma_x(g, h)\gamma_y(g^x, h^x) = \theta_{gh}(x, y)\gamma_{xy}(g, h).
\label{thetagammaid}
\end{eqnarray}
These are the identities required to show that $D^{\w}(G)$ is an algebra and a coalgebra (\cite{DPR}, \cite{MN1}).

\medskip
 Notice that
the restrictions of $\theta_g$ and
$\gamma_g$ to $C(g)$ \emph{coincide}.\ By (\ref{thetaid}) this restriction defines an element in
$Z^2(C(g), \BC^{\times})$.\
In particular, if $g \in Z(G)$ then $\g_g$ is a $2$-cocycle on $G$, and (following \cite{MN2}) we
set
\begin{eqnarray*}
Z_{\w}(G) := \{g \in Z(G) \mid\g_g \in B^2(G, \BC^{\times})\}.
\end{eqnarray*}
This is a subgroup of $Z(G)$.
\medskip

It transpires that most of our considerations concern the multiplicative group of \emph{central
group-like elements}
in $D^{\w}(G)$ and some of its subgroups.\ We review this structure here, following
\cite{MN1}, \cite{MN2}.\ The group of central group-like elements, denoted by $\G_0^{\w}(G)$, may be
described by a short exact sequence
$$
1 \to \widehat G \xrightarrow{\iota} \G_0^\w(G) \xrightarrow{p} Z_{\w}(G) \to 1.
$$
To explain this, fix a family $\t = \{\t_x\}_{x \in Z_\w(G)}$ of normalized 1-cochains on $G$ such that $\delta \t_x
 = \theta_x$ for each $x \in Z_\w(G)$.\ Each element $u \in \G_0^\w(G)$ is uniquely determined by a pair $(\chi, x) \in \widehat G
{\times} Z_\w(G)$ satisfying
$$
 u = \sum_{g \in G} \chi(g) \t_x(g) e_g x\,.
$$
Then $p(u) = x$ and $\iota(\chi) = \sum_{g \in G}\chi (g) e_g 1$.\ With respect to the section of
$p$ defined by $s_\t : Z_\w(G) \to \G_0^\w(G),\ x \mapsto \sum_{g \in G} \t_x(g) e_g x$, the associated
2-cocycle $\b_\t \in Z^2(Z_\w(G), \widehat G)$ is given by
\begin{equation} \label{eq:beta}
\b_\t(x,y)(g) = \theta_g (x,y) \frac{\t_x(g)\t_y(g)}{\t_{xy}(g)}\ \ (x, y\in Z_\w(G), g \in G).
\end{equation}
The assignment $\Lambda : \w \mapsto \b_\t$ defines a group homomorphism 
$H^3(G, \BC^{\times})\rightarrow H^2(Z_\w(G), \widehat G)$.\ We will simply write $\b$ for $\b_\t$ when there is no ambiguity. 

\medskip
For a subgroup $A \subseteq Z_{\w}(G)$ we define $\G_0^\w(G, A)$ by pulling-back along $p$.\ Thus we
have a diagram
$$
 \xymatrix{
 1 \ar[r] & \widehat{G}\ar[d]^-{\id} \ar[r] &\G^\w_0(G, A) \ar[d]^-{} \ar[r]^-{p} &A \ar[d]\ar[r] &
 1\\
 1 \ar[r] & \widehat{G} \ar[r] & \G^{\w}_0(G) \ar[r]^-{p} & Z_{\w}(G)\ar[r] & 1\,.
 }
 $$
The following Proposition is a direct consequence of \cite[Prop.5.2]{MN2}, but we provide a
computational proof for the sake of completeness.
\begin{prop}\label{p:1}
 The equivalence class of the exact sequence
 $$
1 \to \widehat G \xrightarrow{\iota} \G_0^\w(G, A) \xrightarrow{p} A \to 1
$$
is independent of the choices of the representative in the cohomology class of $\w$ and the family
of $1$-cochains $\{\tau_x\}_{x \in A}$ satisfying $ \g_x = \delta \t_x$ for all $x \in 
A$.\ If, in addition, $\BZ_2 \cong A$ and $\exp(\widehat G) \mid 2$, then the 2-cocycle $\b \in Z^2(A, \widehat G)$ 
given by \eqref{eq:beta} is independent of the choices of representative of 
$[\w] \in H^2(G, \BC^\times)$ and the family $\t = \{\t_x\}_{x\in A}$ of $1$-cochains.
\end{prop}
\begin{proof}
Let $\w' = \w \delta f$ where $f$ is a normalized 2-cochain on $G$. Suppose $\theta'$ and $\gamma'$
are defined as above using $\w'$. Then
$$
\theta'_g(x,y) = \theta_g(x,y) \frac{\tilde f(g, xy)}{\tilde f(g,x) \tilde f(g^x, y)}
 $$
 and
$$
\g'_g(x,y) = \g_g(x,y)
\frac{f(y,g)}{f(x y, g)f(x, y)}\frac{f(x^g, y^g)f(g, (xy)^g)}{ f(g,x^g)}\frac{f(x, g)}{f(g, y^g)}
$$
where $\tilde f(g,x) = \frac{f(g,x)}{f(x, g^x)} $.\ In particular, if $x,y \in Z_\w(G) \subseteq
Z(G)$,
then
\begin{eqnarray*}
&&\theta'_g(x,y) = \theta_g(x,y) \delta \tilde f\inv (g, -)(x,y),\\
&& \gamma'_g(x,y) = \gamma_g(x,y) \delta \tilde f\inv (g, -)(x,y)\,.
\end{eqnarray*}

Suppose $\gamma_x = \delta \t_x$ for all $x \in A$.\ Then $\gamma_x' = \delta \t_x \tilde f\inv(x,-)
\a_x$ for some $\a_x \in \widehat G$.\ Therefore,
\begin{eqnarray*}
\beta'(x,y)(g) & = & \theta'_g(x,y) \frac{\t_x(g) \t_y(g)}{\t_{xy}(g)} \frac{\tilde f(xy,g)}{\tilde
f(x,g)\tilde f(y,g)} \frac{\a_x(g) \a_y(g)}{\a_{xy}(g)} \\
&=& \theta_g(x,y) \frac{\tilde f(g,xy)}{\tilde f(g,x) \tilde f(g, y)}\frac{\t_x(g)
\t_y(g)}{\t_{xy}(g)} \frac{\tilde f(xy,g)}{\tilde f(x,g)\tilde f(y,g)} \frac{\a_x(g)
\a_y(g)}{\a_{xy}(g)}\\
 &=& \beta(x,y)(g) \frac{\a_x(g) \a_y(g)}{\a_{xy}(g)}
\end{eqnarray*}
for $x,y \in A$ and $g \in G$.\ If we have $\BZ_2 \cong A = \langle z\rangle$ and 
$\exp(\widehat G) \mid2$, then
$\b(x,y) = \b'(x,y) = 1$ for $(x,y) \ne (z,z)$, and
$$
\b'(z,z)(g) = \b(z,z)(g) \frac{\a_z(g) \a_z(g)}{\a_{1}(g)} = \b(z,z)(g) \a_z(g)^2 = \b(z,z)(g)
$$
for all $g \in G$.
\end{proof}

Let $A \subseteq Z_\w(G)$ be a subgroup and $H \subseteq G$ a subgroup containing $A$.\ Then 
$A \subseteq Z_{\w_H}(H)$, where $\w_H$ is the restriction of $\w$ to $H$, and so we have an exact
sequence
$$
1\to \widehat H \to \G_0^\w(H,A) \to A\to 1\,.
$$
For $u \in \G_0^\w(G, A)$, $u = \sum_{g \in G} \t_x(g) e_g x$ for some $x \in A$ and 
$\t_x \in C^1(G, \BC^\times)$ such that $\delta\t_x = \g_x$.\ Then $\pi_H(u) = \sum_{h \in H} \t_x(h) e_h x \in 
\G_0^{\w_H}(H,A)$, and $\pi_H$ defines a
group homomorphism $\pi_H : \G^\w_0(G, A) \to \G^{\w_H}_0(H, A)$. Moreover, we have following Proposition.

\begin{prop}\label{p:subgroup}
 Suppose $\w$ is a 3-cocycle of $G$, with subgroups $A \subseteq Z_\w(G)$ and $A \subseteq H \subseteq G$.\ 
 Then there is a commutative diagram of exact sequences:
 $$
 \xymatrix{
 1 \ar[r] & \widehat{G}\ar[d]^-{\res} \ar[r] &\G^\w_0(G, A) \ar[d]^-{\pi_H} \ar[r]^-{p} & A
 \ar[d]^-{\id} \ar[r] & 1\\
 1 \ar[r] & \widehat{H} \ar[r] & \G^{\w_H}_0(H, A) \ar[r]^-{p} & A\ar[r] & 1\,.
 }
 $$
\end{prop}
\begin{proof}
The proof is a direct computational consequence of the definitions of $\pi_H$ and $\G^{\w_H}_0(H,
A)$.
\end{proof}

The following standard result will be used repeatedly in our subsequent discussion.
\begin{lem} \label{CE}
 Let $p$ a prime and $P$ a $p$-Sylow subgroup of $G$.\ For any $G$-module $M$ and positive integer $n$, if $H^n(P, M)$ is trivial, then so is $H^n(G, M)_p$.
\end{lem}
\begin{proof}
 By \cite[XII Theorem 10.1]{CE}, the restriction map $\res : H^n(G, M)_p \to H^n(P, M)$ is injective.\ Thus, the Lemma follows immediately.
\end{proof}

\begin{lem}\label{plocalred}  Let $p$ be a prime and let $\w$ be a 3-cocycle of $G$.\ Suppose that $A \subseteq Z_\w(G)$ is a $p$-subgroup and let $P$ be a Sylow $p$-subgroup of $G$.
\begin{enumerate}[label=\rm{(\roman*)}, leftmargin=*]
 \item Assume that the restriction map
$\widehat{G}_p\rightarrow \widehat{P}$ is a split injection.\ Then in order for the exact sequence
\begin{equation}\label{eq:0}
 1 \to \widehat G \to \G^\w_0(G, A) \to A \to 1
\end{equation} 
to split, it is sufficient that 
\begin{equation}\label{eq:1}
1 \to \widehat{P} \to \G^{\w_P}_0(P, A) \to A \to 1
\end{equation} 
is split exact.
\item The
condition that $\widehat{G}_p\rightarrow \widehat{P}$ is a split injection is satisfied if $p=2$ and $P$ contains a unique involution.
\end{enumerate}
\end{lem}
\begin{proof}  First note that $A \subseteq P$.\ We also point out that the restriction map $\widehat{G}_p\rightarrow \widehat{P}$ is \textit{always} an injection.\ This follows from the same Theorem of \cite{CE} cited above, once we remember that $H^1(G, \BC^\times)$ is naturally isomorphic to $\widehat{G}$.
 By Proposition \ref{p:subgroup} there is a commuting
diagram of exact sequences
 \begin{equation*}
 \xymatrix{
 1 \ar[r] & \widehat{G}\ar[d]^-{\res} \ar[r]^-{\iota_0} &\G^\w_0(G, A) \ar[d]^-{\pi_P} \ar[r]^-{p} & A
 \ar[d]^-{\id} \ar[r] & 1\\
 1 \ar[r] & \widehat{P} \ar[r]^-{\iota_2} & \G^{\w_P}_0(P, A) \ar[r]^-{p} & A\ar[r] & 1\,.
 } 
 \end{equation*}
 Since $\res : \widehat{G}_p \rightarrow \widehat{P}$ is injective and $A$ is a $p$-group, the restriction $\pi_P : \ \G^{\w}_0(G, A)_p \to \G^{\w_P}_0(P, A)$ is also injective. Thus, we have the row exact commutative diagram:
 \begin{equation}\label{eq:c1}
 \xymatrix{
 1 \ar[r] & \widehat{G}_p\ar[d]^-{\res} \ar[r]^-{\iota_1} &\G^\w_0(G, A)_p \ar[d]^-{\pi_P} \ar[r]^-{p} & A
 \ar[d]^-{\id} \ar[r] & 1\\
 1 \ar[r] & \widehat{P} \ar[r]^-{\iota_2} & \G^{\w_P}_0(P, A) \ar[r]^-{p} & A\ar[r] & 1\,.
 }
 \end{equation}
 The exact sequence \eqref{eq:0} splits if, and only if, the top row of \eqref{eq:c1} is split exact. 
 \ Thus, if the exact sequence \eqref{eq:1} splits then there exists group homomorphism $j: \G^{\w_P}_0(P, A) \to \widehat{P}$ such that $j \circ \iota_2 =\id_{\widehat{P}}$. Since $\res : \widehat{G}_p \to \widehat{P}$ is assumed to be split injective, there exists a group homomorphism $\kappa: \widehat{P} \to \widehat{G}_p$ such that $\kappa \circ \res = \id_{\widehat{G}_p}$. One can verify directly that $\kappa\circ j\circ \pi_P\circ \iota_1 = \id_{\widehat{G}_p}$, and so the top of \eqref{eq:c1} splits.\ This completes the proof of part (i).
 
 \medskip
 As for part (ii), since $P$ is a $2$-group with a unique involution then $P$ is either cyclic or generalized quaternion.\ As pointed out at the beginning of the proof, we only have to establish the splitting condition.\ We deal with the two cases separately.\ If $P$ is generalized quaternion then $\widehat{P}\cong \BZ_2\times \BZ_2$.\ So in this case, every subgroup splits and we are done.\ 
 
 \medskip
 Now suppose that $P$ is cyclic.\ Here, it is a standard
 consequence of Burnside's normal $p$-complement Theorem \cite[Chapter 7, Theorem 4.3]{G} that $G$ is a semi-direct product $G=Q \rtimes P$ where $Q$ has odd order.\ This makes it clear that $\widehat{G}_2\cong \widehat{P}$ and in particular the splitting condition again holds.\\
 (To verify the conditions needed for Burnside's Theorem (loc.\ cit.) it suffices to show that $N_G(P)=C_G(P)$.\ To see this, notice that $N_G(P)/C_G(P)$ is a faithful group of automorphisms of $P$ having odd order.\ But $P$ is a cyclic $2$-group whence $\Aut(P)$ is a $2$-group. This means that $N_G(P)/C_G(P)$ is trivial, as required.)
\end{proof}

\section{$\w$-admissibility}
In this Section, $\w$ a normalized $3$-cocycle on $G$.

\subsection{Admissible cocycles}
Suppose that $A \subseteq Z(G)$ is a subgroup.\ It is established in \cite{MN2} that in order for $D^{\w}(G, A)$ to
be a quasi-Hopf algebra,
it is necessary and sufficient that the following two conditions hold:
\begin{eqnarray}
&&(a)\ A \subseteq Z_{\w}(G), \label{QHAconds}\\
&&(b)\ 1\rightarrow \widehat{G}\rightarrow \G_0^{\w}(G, A)\rightarrow A\rightarrow 1 \ \mbox{splits}.\notag
\end{eqnarray}
Additional conditions are required in order for
$D^{\w}(G, A)$ to be modular, and we defer discussion of this until Section \ref{Smodularity}.\ Following \cite{MN2}, if (a) and
(b) hold we
say that $A$ is $\w$-\emph{admissible}. 

\medskip
In this case, there exist $\t := \{\t_a\}_{a \in A} \subseteq C^1(G, \BC^\times)$ and $\nu \in C^1(A, \widehat{G})$ such that 
$\delta \t_a = \theta_a$ for all $a \in A$ and $\delta \nu = \b_\t$. We will simply call $(\t, \nu)$ an \emph{$\w$-admissible pair} for $A$.\
In general, there is more than one $\w$-admissible pair for an $\w$-admissible subgroup $A$.\ Any one of them can be used to construct a quasi-Hopf algebra $D^\w(G,A)$.

\medskip
If $A \subseteq H$ is a subgroup of $G$, then it is easily seen that $A$ is also $\w_H$-admissible, where $\w_H := \res_H^G(\w)$.\ Moreover, 
if $(\t, \nu)$ is an $\w$-admissible pair for $A$, then 
$(\t_H, \nu_H)$ is an $\w_H$-admissible pair for $A$, where $\t_{H,a} = \res_H^G(\t_a)$ and $\nu_H(a) = \res_H^G(\nu(a))$ for all $a \in A$.\ We say that $D^{\w_H}(H,A)$ is the generalized twisted double \emph{induced} from $D^\w(G,A)$.


\subsection{Admissibility and groups with one involution}
The main result of this Subsection is the next result.

\begin{thm}\label{thmwadm} If $G$ has a \emph{unique} subgroup $A$
of order $2$, then $A$ is $\w$-admissible for all $\w\in H^3(G, \BC^\times)$.
\end{thm}

\medskip
We must show that (a) and (b) hold.\ To this end, we
establish some preliminary results of independent interest.

\begin{lem}\label{lemma2.3} Let $n \geq 0$ be an integer and let $g \in G$.\ Then
\begin{eqnarray}\label{resform}
\res_{C(g)}^{C(g^n)} [\theta_{g^n}] = [\theta_g]^n.
\end{eqnarray}
In particular, the order of $[\theta_g] \in H^2(C(g), \BC^{\times})$ divides the order
of $g$.
\end{lem}
\begin{proof} We prove (\ref{resform}) by induction on $n$.\ Because $\w$ is normalized then
$\theta_1 = 1$, whence the case $n = 0$ is clear.

\medskip
Take $h := g^n$ and $x, y \in C(g)$ in (\ref{thetagammaid}) to obtain
\begin{eqnarray*}
\theta_g(x, y)\theta_{g^n}(x, y)\gamma_x(g, g^n)\gamma_y(g, g^n) = \theta_{g^{n+1}}(x, y)\gamma_{xy}(g, g^n).
\end{eqnarray*}
Defining a $1$-cochain $f$ on $C(g)$ as
$f(x) := \gamma_x(g, g^n)$, the last display just says that
\begin{eqnarray*}
\theta_g(x, y)\theta_{g^n}(x, y) \delta f(x, y) = \theta_{g^{n+1}}(x, y).
\end{eqnarray*}
Using the inductive hypothesis, it then follows that
\begin{eqnarray*}
\res_{C(g)}^{C(g^{n+1})} [\theta_{g^{n+1}}] = [\theta_g] \res_{C(g)}^{C(g^n)} [\theta_{g^n}] = 
[\theta_g] [\theta_g]^n = [\theta_g]^{n+1}.
\end{eqnarray*}
This completes the proof of (\ref{resform}).\
To prove the last statement of the Lemma, take $n$ to be the order of $g$.\ Because $\theta_1 = 1$
we conclude from (\ref{resform}) that $[\theta_g]^n = 1$, as required.
\end{proof}

\medskip
We can now prove
\begin{thm}\label{thmZw} If $B \subseteq Z(G)$ is a subgroup that satisfies $\gcd(|B|, |H^2(G, \BC^{\times})|) = 1$, then
 $B \subseteq Z_{\omega}(G)$.
\end{thm}
\begin{proof} We have to show that every $g \in B$ lies in $Z_{\w}(G)$.\
Because $B \subseteq Z(G)$, what we have to establish is that $\theta_g \in B^2(G, \BC^{\times})$.\
However, by Lemma \ref{lemma2.3} we know that the order of $[\theta_g] \in H^2(G, \BC^{\times})$
divides
the order of $g$.\ On the other hand, the hypothesis of the Theorem implies that the order of
$[\theta_g]$ is \emph{coprime} to the order of $g$.\ Therefore $[\theta_g]$ has order $1$, which
means
exactly that $\theta_g \in B^2(G, \BC^{\times})$.\ This completes the proof of Theorem \ref{thmZw}.
\end{proof}
\begin{cor}\label{cora} If $G$ has a unique subgroup $A$ of prime order $p$ and $A \subseteq Z(G)$,
then $A \subseteq Z_{\w}(G)$
\end{cor}
\begin{proof}
Because $G$ has a unique subgroup of prime order $p$ then a Sylow $p$-subgroup $P$ of $G$ has the
same property,
and it is well-known (\cite[Theorem 4.10(ii)]{G}) that in this situation either $P$ is cyclic or
$p = 2$ and $P$ is generalized quaternion.\ In either case, $H^2(P, \BC^{\times})$ is trivial.\ (This holds whenever $P$ has periodic cohomology.)\ It follows from Lemma \ref{CE} that $H^2(G, \BC^\times)_p = 1$, and the Corollary then follows
from Theorem \ref{thmZw}.
\end{proof}

\medskip
Having dealt with condition (\ref{QHAconds})(a), we turn to the question of the splitting of the sequence in (\ref{QHAconds})(b).\
We will prove

\begin{thm}\label{thmsplit1} If $G$ has a \emph{unique} subgroup $A$ of order $2$ then the
short exact sequence
$1 \to \widehat G \to \G^\w_0(G, A) \to A \to 1$ splits.
\end{thm}

\begin{remark} This result is generally \emph{false} if we replace $2$ by an \emph{odd} prime.
\end{remark}

\begin{proof}
After Lemma \ref{plocalred}, the Theorem will follow in general if we can prove it for a Sylow
$2$-subgroup of $G$.\ So for the remainder of the proof we assume that $G$ is a $2$-group.\ Thus, as before, $G$
is either \emph{cyclic}
or \emph{generalized quaternion}.

\medskip
Suppose first that $G$ is cyclic.\ Then $B^{\w} = Z_\w(G)= G$, so $D^{\w}(G)$ is abelian by \cite[Cor. 3.6]{MN1}.\ In particular, by Theorem 8.5 (loc.\ cit) it follows that
$\G^{\w}_0(G)$ is a direct sum of \emph{two} (nontrivial) cyclic $2$-groups.\ If the short exact sequence in the statement does not split, then $\G_0^\w(G,A)$ is a cyclic 2-group since $H^2(A, \widehat{G}) \cong \BZ_2$. This forces $\G_0^\w(G)$ to be cyclic, a contradiction.

\medskip
Now assume that $G$ is generalized quaternion.\ Thus
\begin{eqnarray*}
G = \langle r, s \mid r^n = s^2, s^4 = 1, srs\inv = r\inv\rangle
\end{eqnarray*}
where $|G| = 4n$ and $n > 1$ is a power of $2$.\ We have $Z(G)=\langle z\rangle = A$ where $z = s^2$, and
$\G_0^{\w}(G, A) = \G^{\w}_0(G)$ since $H^2(G, \BC^\times)$ is trivial.\ Moreover,
 $G/G'\cong \BZ_2\times \BZ_2$ is generated by $r G', s G'$.

 \medskip
 Let $\t_z$ be a fixed normalized 1-cochain such that $\delta \t_z = \g_z$. Then
 $$
 \b(z,z)(g) = \theta_g(z,z) \t_z (g)^2
 $$
 is a 2-cocycle in $Z^2(A, \widehat{G})$ associated to the extension $1 \to \widehat G \to \G^\w_0(G, A) \to 
 A \to 1$.

 \medskip
 Step (I).\ $\b(z,z)(r) = 1$.\ To prove this we introduce the group $E := \langle \ol r, \ol
 s\rangle$, which is a generalized quaternion group of order $8n$ in which $G$ is identified with
 the subgroup of index $2$ generated by $r:=\ol r^2, s:=\ol s$.\ Since $E$ has periodic
 cohomology, the Artin-Tate theory tells us that $\res_G^E : H^3(E, \BC^\times) \to H^3(G, \BC^\times)$ is an \emph{epimorphism}.\
 Hence, there exists $\ol \w \in Z^3(E, \BC^\times)$ such that $\ol \w_G = \w$.

\medskip
 For $g \in E$, let $\ol \theta_g$ and $\ol\g_g$ be functions associated with $\ol \w$ given by
 \eqref{eq:01} and \eqref{eq:02} and $\ol \t_z $ a normalized 1-cochain of $E$ such that $\delta
 \ol \t_z = \ol \g_{z}$.\ Then,
 $$
 \ol \b(z, z)(g) = \ol \theta_g(z, z) \ol \t_z(g)^2
 $$
 defines a 2-cocycle in $Z^2(A, \widehat {E})$.\ Since $\exp(\widehat E) = 2$ then $\ol \b(z, z)(g^2) = 1$ for
 all $g \in E$.\ In particular, $\ol \b(z, z)(r) = 1$.\ Note that $\delta \res^E_G(\ol \t_z) = \ol
 \g_{z}|_G = \g_{z}$.\ We then find that\
 $$
 1 = \ol \b(z, z)(r) = \ol \theta_r(z, z) \ol \t_z(r)^2 = \theta_r(z, z) \t_z(r)^2 = \b(z, z)(r)
 $$
 by Proposition \ref{p:1}.

 \medskip
 Step (II).\ $\b(z,z)(s) = 1$.\ Let $Q := \langle s, r^{n/2}\rangle \cong Q_8$.\ Applying the same
 argument as Step (I) to $Q$,
 we obtain $\b(z,z)(s) = 1$.

 \medskip
 Step (III).\ Since $\b(z,z)(s) = \b(z,z)(r) = 1$, then also $\b(z,z) = 1$.\ Therefore, the sequence in
 the statement of the Theorem splits, and the proof of the Theorem is complete.
\end{proof}

Combining Theorem \ref{thmsplit1} and the case $p = 2$ of Corollary \ref{cora} implies
Theorem \ref{thmwadm}.


\subsection{Admissibility and representation groups}\label{SSrepgp}
In this Subsection we consider $D^{\omega}(G, A)$ in the case that $G$ is a
\emph{perfect representation group}.\ The main result is Theorem \ref{thmE} below.\ It provides us with many cases when
$D^{\omega}(G, A)$ is a quasi-Hopf algebra, and in particular provides a somewhat different approach to some of the examples covered
by Theorem \ref{thmwadm}.\ We will discuss some of these in Subsection \ref{SSex}

\medskip
 Recall that $G$ is called \emph{perfect} if it coincides with its commutator subgroup,
$G = G'$.\ Alternatively, $G$ has only one (1-dimensional) character, i.e., $\widehat{G} = 1$.\ The theory of representation groups
was originally developed by Schur.\ An exposition can be found in \cite{CR}, $\S$11E.\ We develop some of the background that we need.

\medskip
\medskip
For a finite group $G$, a central extension $E$ of $G$
\begin{equation}\label{unicentext}
 1 \to A \to E \xrightarrow{p} G\to 1
\end{equation}
 is called a \emph{stem extension} if $E$ is finite and $A \subseteq E'$. There exists a stem extension $E$ with the largest order, and $E$ is called a \emph{representation group or Schur covering group} of $G$. In this case, 
$$
A = H^2(G, \BC^{\times}).
$$
It is easy to see that a representation group $E$ of $G$ is perfect if, and only if, $G$ is perfect.
 In particular, if $G$ is perfect then it has a unique representation group (up to isomorphism) and the representation group is also perfect.

\medskip
The following result is more-or-less implicit in the presentation of \cite{CR}, and in any case it is well-known.
\begin{prop}\label{propCR} Suppose that $E$ is the representation group of a perfect finite group $G$.\ Then $H^2(E, \BC^{\times}) = 1$.
\end{prop}
\begin{proof}

\medskip
Let $C := H^2(E, \BC^{\times})$, with $F$ the representation group of $E$.\ It occurs in the perfect central extension given in the middle row of the diagram below.\ 

\medskip
We have a diagram (using previous notation)
 \begin{eqnarray*}
\xymatrix{
&&1\ar[r] &A\ar[r]&E\ar[r]^p\ar[d]^{\id}&G\ar[r]&1\\
&1\ar[r]&C\ar[d]\ar[r]&F\ar[d]^{\id}\ar[r]^q &E\ar[d]^p\ar[r]&1\\
&1\ar[r]&B\ar[r]&F\ar[r]&G\ar[r]&1\\
 }
\end{eqnarray*}
where $B := q\inv(A)$.\ Note that $A \subseteq Z(E)$ and we have $[F, B] \subseteq C$.\ 
Then because $C \subseteq Z(F)$ we obtain $B \subseteq Z_2(F)$, the subgroup of $F$ such that $Z_2(F)/Z(F)=Z(F/Z(F))$.\ Since $F$ is perfect, by Gr\"un's Lemma we have $Z_2(F) = Z(F)$.\ Therefore $B \subseteq Z(F)$.\
Thus we have shown that 
the bottom row or $F$ is a stem extension of $G$. Since $|F| \ge |E|$, $|F|=|E|$.\ In particular, $F$ is a representation group of $G$. Therefore, $F \cong E$ and so $C$ must be trivial. 
\end{proof}

\medskip
We can now prove
\begin{thm}\label{thmE}Suppose that $G$ is a perfect group and let $E$ be a representation group of $G$, so that
(\ref{unicentext}) is the universal central extension of $G$.\ Then $D^{\omega}(E, B)$ is a quasi-Hopf algebra for every normalized $3$-cocycle $\omega$ on $E$ and for every subgroup $B \subseteq Z(E)$.\ That is, every subgroup $B$ of $Z(E)$ is $\omega$-admissible.
\end{thm}
\begin{proof} We have $\widehat{E} = 1$, so the short exact sequence $1\rightarrow \widehat{E}\rightarrow \G_0^{\w}(E, B)\rightarrow B\rightarrow 1$ splits for trivial reasons.\ We assert that $Z(E) \subseteq Z_{\omega}(E)$.\ For this we must show that
for each $g \in Z(E)$ we have $\theta_g \in B^2(E, \BC^{\times})$.\ But this follows immediately from Proposition \ref{propCR}.\
Thus, $B$ is $\w$-admissible, and the Theorem is proved.
\end{proof}

\subsection{Examples}\label{SSex} The group-theoretic classification of finite groups with a unique involution is essentially a consequence of the Brauer-Suzuki
Theorem that we mentioned in the Introduction.

\medskip
Let $G$ be a group with a unique subgroup $A$ of order $2$.\ Let $Q := O(G)$ be the largest normal subgroup of $G$ with
\emph{odd order}, and let $P$ be a Sylow 2-subgroup of $G$.\ We always have $A\subseteq P$.\ To describe the possible groups $G$ it is useful to sort them into
classes according to (i) whether $G$ is solvable or nonsolvable, and (ii) whether $P$ is cyclic or generalized quaternion.\ For later purposes
we single out the case when $P = A$, but generally we will suppress some of the details in other cases - they will not be needed.

\medskip\noindent
1.\ $P = A$.\ Then $G = Q \times A$.

\medskip\noindent
2.\ $P$ is cyclic.\ Then $G = Q \rtimes P$.

\medskip\noindent
3.\ $P$ is a generalized quaternion and $G$ is solvable.\ Then either $G = Q{\rtimes} P$ or else $|P|\leq 16$
and $G \cong Q{\rtimes}SL_2(3)$ or $Q{\rtimes}SL_2(3){\cdot}2$.\ ($SL_2(3)$ and $SL_2(3){\cdot}2$ are the binary tetrahedral and octahedral groups respectively.)

\medskip\noindent
4.\ $P$ is a generalized quaternion, $G$ is perfect and $Q \subseteq Z(G)$.\ Then either $G \cong SL_2(q)$ for some odd prime power $q\geq 5$, or
$G$ is a (nonsplit) central extension $m{\cdot}A_n$ of the alternating group $A_n$ by a cyclic group of order $m$ for $m \in \{2, 6\}$ and $n \in \{6, 7\}$.\ (In fact $2{\cdot}A_6 \cong SL_2(9)$, and there are no other 
isomorphisms among these groups.)

\medskip
The binary polyhedral groups (finite subgroups of $SU(2)$) implicitly occur in this list.\ All of them are solvable except for
the binary icosahedral group, which is isomorphic to $SL_2(5)$.

\medskip
With the exception of $q = 9$, $SL_2(q) (q\geq 5)$ is the representation group of the simple group $PSL_2(q)$.\ On the other hand
$6{\cdot}A_n (n = 6, 7)$ is the representation group
of $A_n$.

\medskip
Thus, for example, if $G = SL_2(q)\ (q\geq 5)$ or $G = m{\cdot}A_n$, we obtain quasi-Hopf algebras $D^{\omega}(G, B)$ for all subgroups
$B \subseteq Z(G)$ and all $3$-cocycles $\omega$.\ This follows from Theorems \ref{thmwadm} or Theorem \ref{thmE}.


\section{Modularity and Super-Modularity}\label{Smodularity}
In this Section we complete the proof of the main Theorems stated in the Introduction.\ We first recall the definitions of modularity and develop some Lemmas based on \cite{MN2}. We prove more precise versions of the main Theorems. 

\subsection{Modular and Super-modular tensor categories}
Let $\CC$ be a braided fusion category with braiding $c$ and unit object $\1$, and let $\Irr(\CC)$ denote the set of isomorphism classes of simple objects in $\CC$.    \ Generally we do not distinguish between a simple object and its \emph{isomorphism class}. Throughout this paper, all fusion categories $\CC$ are pseudounitary and hence they are spherical with respect to the canonical pivotal structure. In particular, the categorical dimension of any nonzero $V \in \CC$ is positive. (cf. \cite{ENO}). 

\medskip
Let $\AA$ be a fusion subcategory of $\CC$.\ Recall that the \emph{M\"uger centralizer} of $\AA$ in $\CC$ (over $\BC$) is the full subcategory with the objects given by
$$
C_\CC(\AA) := \{X \in \CC \mid c_{Y,X} {\circ} c_{X,Y} = \id_{X {\o} Y} \text{ for all } Y \in \AA\}\,.
$$ 
A simple object $f$ of $\CC$ is called a \emph{fermion} if $f {\o} f \cong \1$ and $c_{f,f} = {-}\id_{f \o f}$.
In this context, we can now define modularity and super-modularity. 

\medskip
The braided fusion category $\CC$ is called \emph{modular} if 
$$ 
\{X \in \irr(\CC) \mid c_{Y,X} {\circ} c_{X,Y} = \id_{X \o Y} \text{ for all } Y \in \irr(\CC)\} = \{\1\} .
$$
Following \cite{BGHN}, $\CC$ is called \emph{super-modular } if 
$$ 
\{X \in \irr(\CC) \mid c_{Y,X} {\circ} c_{X,Y} = \id_{X \o Y} \text{ for all } Y \in \irr(\CC)\} = \{\1, f\} 
$$
for some fermion $f$ of $\CC$.\ In particular, $\CC$ is a super-modular category if, and only if, $C_\CC(\CC)$ is braided tensor equivalent to the category 
$\SVec$ of super vector spaces over $\CC$.\ Note that a super-modular category is called a \emph{slightly degenerate modular category} in \cite{DGNO}.\
 A quasitriangular quasi-Hopf algebra $H$ is called \emph{modular} or \emph{super-modular} if $\R(H)$ is modular or super-modular respectively.

\subsection{Modularity of $D^\w(G,A)$}
Let $\w$ be a normalized 3-cocycle of $G$, and $A$ an $\w$-admissible subgroup of $Z(G)$.\ Then the quasi-Hopf algebra $D^\w(G,A)$ is defined and is a homomorphic image of $D^\w(G)$.\ In particular, $\CC := \R(D^\w(G))$ is a modular tensor category and $\BB := \R( D^\w(G,A))$ is a (braided) fusion subcategory of $\CC$.

\medskip
Suppose $(\t, \nu)$ is an $\w$-admissible pair of $A$. The map
 $$
 \psi_{\t, \nu}(a) : e_g x \mapsto \delta_{a,g}\frac{\t_a(x)}{\nu(a)(x)}\ \ (g, x \in G, a \in A)
$$ 
defines a 1-dimensional character of $D^\w(G)$. We let $\widehat a$ denote the isomorphism class of 1-dimensional representations of 
$D^\w(G)$ which afford the character $\psi_{\t, \nu}(a)$. 

\medskip
The set of isomorphism classes of 1-dimensional representation of $D^\w(G)$ form a group, denoted by\footnote{SC stands for `simple currents'} $\SC(G,\w)$, under the tensor product of 
$\CC$, and 
$$
\tilde{p}_\nu : A \to \SC(G,\w),\\ a \mapsto \widehat{a}
$$ defines an injective group homomorphism.\ The braided monoidal structure on $\CC$ induces a braided monoidal structure on the full $\BC$-linear subcategory $\AA$ of $\CC$ generated by $\tilde{p}_\nu(A)$.\ In particular, 
$A \cong \irr(\AA)$ as groups.\ It is worth noting that $\AA$ and $\R(A)$ are equivalent $\BC$-linear abelian categories and they have the same fusion rules, but they may not be equivalent as tensor categories.

\medskip
The associativity and braiding on $\AA$ furnish an Eilenberg-MacLane 3-cocycle on $A$ via $\tilde{p}_\nu$ 
 which is given by 
$$
(\res_A^G \w, d_{\t, \nu}),\ \ d_{\t, \nu}(a,b) = \frac{\t_b(a)}{\nu(b)(a)}
$$ 
where $d_{\t, \nu}(a,b)$ is the scalar determined the braiding $\widehat{a} {\o} \widehat{b} \to \widehat{b} {\o} \widehat{a} $. By \cite{EM},
 $$
 q_{\t, \nu}(a) := d_{\t, \nu}(a,a)
 $$ 
 is a quadratic form on $A$ and it uniquely determines the (Eilenberg-MacLane) cohomology class of $(\res^G_A \w, d_{\t, \nu})$. Then
 \begin{equation}\label{eq:bicharacter}
 (a,b)_{\t, \nu} = \delta q_{\t, \nu} (a,b) = \frac{\t_b(a)\t_a(b)}{\nu(b)(a)\nu(a)(b)} \ \ (a, b \in A).
 \end{equation}
 defines a symmetric bicharacter of $A$.\ Moreover, the $S$-matrix of $\AA$ is given by
$$
 \ol S_{\widehat{a}, \widehat{b}} := d_{\t, \nu}(a,b) \cdot d_{\t, \nu}(b,a) = (a,b)_{\t, \nu} \ \ \ (a, b \in A).
$$
Therefore, $\AA$ is a modular subcategory of $\CC$ if, and only if, the symmetric bicharacter $(\cdot, \cdot)_{\t, \nu}$ is nondegenerate. 
It has also been proved \cite[Theorem 5.5]{MN2} that 
$$
C_\CC(\AA) = \BB \quad\text{and}\quad C_\CC(\BB) = \AA, 
$$ 
and the modularity of $\BB$, as well as $\AA$, is determined by the nondegeneracy of the bicharacter $(\cdot, \cdot)_{\t, \nu}$. 

\begin{remark}
If the bicharacter $(\cdot, \cdot)_{\t, \nu}$ of $A$ is totally degenerate, i.e., $(a, b)_{\t, \nu} = 1$ for all $a, b \in A$, then $\AA$ is a symmetric fusion category. It follows from a theorem of Deligne (cf. \cite{O}) that there are only two possibilities in this situation:
\begin{enumerate}
\item[\rm (i)] $\AA$ is equivalent to $\Rep(A)$ as braided tensor category. In this case, $\AA$ is called \emph{Tannakian}.
\item[\rm (ii)] $\AA$ is equivalent to $\Rep(A, u)$ as braided fusion category, where $\Rep(A, u)$ is the braided fusion category $\Rep(A)$ equipped with a symmetric braiding given by the order 2 element $u \in A$. In this case, $\AA$ is called \emph{super-Tannakian}.
\end{enumerate} 

The question of whether $\AA$ is modular, super-Tannakian or Tannakian is completely determined by the quadratic form $q_{\t, \nu} : A \to \BC^\times$.\
Indeed, $\AA$ is modular (resp.\ Tannakian) if, and only if, the quadratic form $q_{\t, \nu}$ is nondegenerate (resp. trivial), while $\AA$ is super-Tannakian if $q_{\t, \nu}$ is an order 2 character of $A$.

\medskip
In particular, if $A:=\langle a \rangle$ has order 2, then whether $\AA$ is modular, super-Tannakian or Tannakian is determined by the value $q_{\t, \nu}(a)= \pm i, -1$ and $1$ respectively. In particular, $\AA$ is super-Tannakian if, and only if, $\AA$ is equivalent to $\SVec$ as braided tensor categories. 
\end{remark}

 It is possible that the bicharacter $(\cdot,\cdot)_{\t,\nu}$ is \emph{independent of $\tau$ and $\nu$}.\ We illustrate this possibility in the following Example and Proposition. 
\begin{example} \label{ExE}
Let $E$ be the representation group of a perfect group $G$ and $\w$ a normalized 3-cocycle on $E$.\ Note that if
$A \subseteq Z(E)$, then $A$ is $\w$-admissible for all $\w \in Z^3(E, \BC^\times)$ by Theorem \ref{thmE}.\ Since $E$ is also perfect, $\widehat{E}$ is trivial. Therefore, $C^1(A, \widehat{E}) = 1$, and there exists a unique family $\t = \{\t_a\}_{a \in A} \subseteq C^1(E, \BC^\times)$ such that $\delta \t_a = \theta_a$ for all $a \in A$.\ Therefore, there is only one $\w$-admissibility pair $(\t, \nu)$ for $A$, where $\nu$ is unique element of $C^1(A, \widehat{E})$, and its associated bicharacter is given by
$$
(a,b)_{\t,\nu} = \frac{\t_a(b)\t_b(a)}{\nu(a)(b)\nu(b)(a)} = \t_a(b)\t_b(a) \quad (a,b \in A).
$$ 
\end{example}

\medskip
\begin{prop} \label{p:nondegenrate} Let $\w$ be a normalized 3-cocycle on $G$.\ Suppose $A := \langle a \rangle \subseteq G$ is an $\w$-admissible subgroup of order 2 and $(\t,\nu)$ an $\w$-admissible pair of $A$. Then 
\begin{equation} \label{eq:b2}
(x,y)_{\t, \nu} = \frac{\t_x(y)\t_y(x)}{\nu(x)(y)\nu(y)(x)} = \left\{\begin{array}{ll} \t_a(a)^2 & \text{if } x = y = a,\\
 1 & \text{otherwise}. 
 \end{array}\right. 
\end{equation}
In particular, $(\cdot,\cdot)_{\t, \nu}$ is nondegenerate (resp. totally degenerate) on $A$ if, and only if, $\t_a(a)^2 = {-}1$ (resp. $\t_a(a)^2 = 1$).\
 Moreover, $D^\w(G,A)$ is modular if, and only if, $\res_A^G(\w)$ is a not coboundary of $A$. 
\end{prop}
\begin{proof}
Since $\nu \in C^1(A, \widehat{G})$ is normalized, $\nu(x)(y) \nu(y)(x) = 1$ all $x,y \in A$.\ Thus, \eqref{eq:b2} follows immediately from this observation.\
 Since $(\cdot,\cdot)_{\t,\nu}$ is a bicharacter of $A$, $(a,a)_{\t,\nu} = \pm 1$ which completely determines the nondegeneracy of $(\cdot,\cdot)_{\t,\nu}$.\ The second statement of the Proposition follows. Note that the value $\w(a,a,a)$ depends on the cohomology class of $\res_A^G(\w)$ of $A$. Since
$$
\t_a(a)^2 = \theta_a(a,a) = \w(a,a,a)\,,
$$
the bicharacter $(\cdot,\cdot)_{\t,\nu}$ on $A$ is nondegenerate if, and only if, $\w(a,a,a) = {-}1$ which is equivalent to that $\res_A^G(\w)$ is not a coboundary of $A$.\ The last statement now follows directly from \cite[Theorem 5.5]{MN2}.
\end{proof}

The following proposition addresses the super-modularity of $D^\w(G,A)$ when $A \cong \BZ_2$.

\begin{prop} \label{p:degenerate} Let $A \subseteq Z(G)$ be an $\w$-admissible subgroup for some normalized $3$-cocycle $\w$ on $G$, and let
$(\t, \nu)$ be an $\w$-admissible pair for $A$.\ Set $\CC := \Rep(D^{\omega}(G)),\ \BB := \Rep(D^{\omega}(G, A)) $ and $\AA := C_{\CC}(\BB)$.\
Then the following hold : 
\begin{enumerate}
\item[{\rm (i)}] 
The bicharacter $(\cdot, \cdot)_{\t, \nu}$ of $A$ is totally degenerate if, and only if, 
$$
C_\BB(\BB) = \AA.
$$
In this case, $\res_A^G(\w)$ is a coboundary of $A$.
\item[{\rm (ii)}]Suppose $A := \langle a\rangle$ has order $2$. Then the following statements are equivalent: 
\begin{enumerate}
\item[\rm (a)] $\BB$ is super-modular.
\item[\rm (b)] $\AA$ is equivalent to $\SVec$ as braided tensor categories.
\item[\rm (c)] $d_{\t, \nu}(a,a) = \frac{\t_a(a)}{\nu(a)(a)} = {-}1$. 
\end{enumerate}
\end{enumerate}
\end{prop} 
\begin{proof}

\noindent (i) The totally degeneracy of $(\cdot, \cdot)_{\t, \nu}$ on $A$ implies $\AA$ is a symmetric fusion category, i.e. $C_\AA(\AA) = \AA$. By a theorem of Delign (cf. \cite{O}), $\AA$ is tensor equivalent to $\Rep(K)$ for some finite group $K$. It is immediate to see that $K \cong A$. Since the associativity on $\AA$ is given by $\res_A^G(\w\inv)$, $\res_A^G(\w\inv)$ must be a coboundary.

Since $C_\AA(\AA)$ is a fusion subcategory of $C_\CC(\AA)$, by \cite[Theorem 5.5]{MN2}, we have $\AA$ is a fusion subscategory of $\BB$, and so 
$\AA$ is also a fusion subscategory of $C_\BB(\BB)$. Since $C_\BB(\BB)$ is a fusion subscategory $C_\CC(\BB)=\AA$, we have $C_\BB(\BB)=\AA$. 

Conversely, if $C_\BB(\BB)=\AA$, then $\AA$ is a braided fusion subcategory of $\BB$ and $\AA$ is symmetric. In particular, 
$c_{\widehat{b}, \widehat{a}} \circ c_{\widehat{a}, \widehat{b}} = \id_{\widehat{a}\o\widehat{b}}$ for all $\widehat{a}, \widehat{b} \in \irr(\AA)$. In terms of the Eilenberg-MacLane 3-cocycle $(\res_A^G(\w), d_{\t, \nu})$ of $A$, we have
$$
1= d_{\t, \nu}(a, b) d_{\t, \nu}(b,a) = (a,b)_{\t, \nu} \quad\text{for } a, b \in A\,.
$$

\noindent (ii) (a) $\Leftrightarrow$ (b): If $\BB$ is super-modular, then $C_\BB(\BB)$ is a braided fusion subcategory of $\BB$ equivalent $\SVec$. In particular, $\FPdim( C_\BB(\BB)) = 2$. Since
$C_\BB(\BB)$ is a braided fusion subcategory of $C_\CC(\BB) = \AA$ and $\FPdim(\AA) = 2$, $C_\BB(\BB) = \AA$. Conversely, if $\AA$ is braided tensor equivalent to $\SVec$, then each entry of the $S$-matrix of $\AA$ is 1 and hence the bicharacter $(\cdot, \cdot)_{\t, \nu}$ on $A$ is totally degenerate. By (i), $C_\BB(\BB) = \AA$ and so $\BB$ is super-modular. 
 
\noindent (b) $\Leftrightarrow$ (c): $\AA$ is equivalent to the category $\SVec$ if, and only if, the nonunit simple object $\widehat{a}$ is a fermion, i.e. $c_{\widehat{a}, \widehat{a}} = {-}\id_{\widehat{a} \o \widehat{a} }$.\ In term of the Eilenberg-MacLane 3-cocycle $(\res_A^G\w, d_{\t, \nu})$ of $A$, the last equality is equivalent to
$$
-1 = d_{\t, \nu}(a, a) = \frac{\t_a(a)}{\nu(a)(a)}\,. \qedhere
$$
\end{proof}

Recall from Section 3 that if $A$ is an $\w$-admissible subgroup of $G$, then $A$ is an $\w_H$-admissible subgroup of $H$ for any subgroup $H$ of $G$ containing $A$. The following lemma shows that $D^\w(G,A)$ and the induced $D^{\w_H}(H,A)$ have the same modularity or super-modularity.

\begin{lem}\label{l:induced}
Let $G$ be a group with subgroups $A \subseteq H \subseteq G$, and let $\w$ be a normalized 3-cocycle on $G$ such that $A$ is $\omega$-admissible. Set $\CC_G = \Rep(D^\w(G))$, $\CC_H = \Rep(D^{\w_H}(H))$, $\BB_G = D^\w(G,A)$ and $\BB_H = D^{\w_H}(H,A)$, where $D^{\w_H}(H,A)$ is induced from $D^\w(G,A)$.\ Then, $C_{\CC_G}(\BB_G)$ is equivalent to $C_{\CC_H}(\BB_H)$ as braided tensor categories.\ In particular, $\BB_G$ is modular if, and only if, $\BB_H$ is modular. If $A$ is of order 2, then $\BB_G$ is super-modular if, and only if, $\BB_H$ is super-modular.
\end{lem}
\begin{proof}
Let $(\t,\nu)$ be an $\w$-admissible pair of $A$ for $D^\w(G,A)$, and let $(\t_H,\nu_H)$ be the $\w_H$-admissibility pair of $A$ for $D^{\w_H}(H,A)$ induced from $D^\w(G,A)$ (cf. Section 3). Let $\AA_G = C_{\CC_G}(\BB_G)$ and $\AA_H = C_{\CC_H}(\BB_H)$. By \cite[Theorem 5.5]{MN2}, $\AA_G$ and $\AA_H$ are pointed braided fusion categories determined by quadratic forms $q_{\t, \nu}:A \to \BC^\times$ and $q_{\t_H, \nu_H}:A \to \BC^\times$ respectively.\ However, for $a \in A$,
 $$
 q_{\t_H, \nu_H}(a) = \frac{\t_{H,a}(a)^2}{\nu_H(a)(a)^2} = \frac{\t_{a}(a)^2}{\nu(a)(a)^2} = q_{\t, \nu}(a)\,.
 $$
 Therefore, $\AA_G$ and $\AA_H$ are equivalent as braided tensor categories. In particular, $\BB_G$ is modular if, and only if, $\BB_H$ is modular by \cite[Theorem 5.5]{MN2}. Similarly, if $|A|=2$ then $\BB_G$ is super-modular if, and only if, $\AA_G$ and $\AA_H$ are equivalent to $\SVec$ as braided tensor categories. By Proposition \ref{p:degenerate}, this is equivalent to the super-modularity of $\BB_H$.
 \end{proof}
 
Note that $\b$ is defined in terms of the family $\t = \{\t_a\}_{a\in A}$, which is quite arbitrary, and the choice of $\nu$ depends on $\t$. We will denote $\b$ as $\b_\t$ in the following Lemma which describes the relation of these two parameters. 

\begin{lem} \label{l:relation}
Let $\w$ be a normalized 3-cocycle of $G$ and $A \subseteq Z(G)$ an $\w$-admissible subgroup.\ Suppose $(\t, \nu)$ is an $\w$-admissible pair of $A$.\ Then we have
\begin{equation}
\theta_x(a,b) = \frac{\t_{ab}(x)}{\t_a(x) \t_b(x)} \frac{\nu(a)(x)\nu(b)(x)}{\nu(ab)(x)}\ \ (x \in G, a,b \in A).
\end{equation}
If $(\t', \nu')$ is another $\w$-admissible pair of $A$, then there exists $f \in Z^1(A, \widehat{G})$ such that 
\begin{equation}
\frac{\t'_a(x)}{\nu'(a)(x)} = \frac{\t_a(x)}{\nu(a)(x)}\cdot f(a)(x)\,.
\end{equation}
\end{lem} 
\begin{proof}
The first equality follows immediately from the definition of $\b_\t$ and that $A$ is $\w$-admissible. By definition, 
$$
\b_{\t'}(a,b)(x) = \theta_x(a,b)\frac{\t'_a(x)\t'_b(x)}{\t'_{ab}(x)}
$$
and $\t'_a = \t_a \chi(a)$ for some $\chi \in C^1(A,\widehat{G})$ for each $a \in A$. Thus, we have
$$
\b_{\t'}(a,b)(x) = \theta_x(a,b)\frac{\t_a(x)\t_b(x)}{\t_{ab}(x)} \frac{\chi(a)(x)\chi(b)(x)}{\chi(ab)(x)} = \frac{\nu'(a)(x)\nu'(b)(x)}{\nu'(ab)(x)} 
$$
for all $a,b \in A$ and $x \in G$. This implies
$\delta (\nu' \chi\inv) = \delta \nu$.
Therefore, $\nu' \chi\inv f = \nu$ for some $f \in Z^1(A, \widehat{G})$ and hence
$$
\frac{\t'_a(x)}{\nu'(a)(x)} = \frac{\t_a(x)}{\nu(a)(x)} \cdot f(a)(x)\quad\text{for all} a,b \in A, x \in G\,. \qedhere
$$ 
\end{proof}

\subsection{Proof of the Main Theorem}
We now prove a
precise version of our first main theorem based on the results of \cite{MN2}.
\begin{thm}\label{thmmodular} Suppose that $G$ is a finite group with a unique subgroup $A$ of order
$2$. Then for any normalized $3$-cocycle
$\w$ on $G$, $A$ is $\w$-admissible, i.e. there is a commuting diagram of quasi-Hopf algebras and morphisms
 $$\
 \xymatrix{
 \BC_{\w}^{G}\ar[d]^-{id} \ar[r]^-{i} & D^{\w}(G) \ar[d]^-{\pi} \ar[r]^-{p} & \BC G
 \ar[d]^-{\pi_{\bar{G}}} \\
 \BC_{\w}^{G} \ar[r]^-{i} & D^{\w}(G, A) \ar[r]^-{p'} & \BC(\bar{G})\,.
 }
 $$
 Moreover, $D^{\w}(G, A)$ is a modular quasi-Hopf algebra if, and only if, $[\w]$ contains a
 $2$-generator of $H^3(G, \BC^{\times})$. Exactly one half of $[\w] \in H^3(G, \BC^\times)$such that quasi-Hopf algebras $D^\w(G,A)$ are modular.
 \end{thm}
\begin{proof} The existence of the diagram is equivalent to the $\w$-admissibility of $A$ \cite{MN2},
which we have already established in Theorem \ref{thmsplit1} for all $\w$. So it remains to determine which $\w$ lead to a
modular quasi-Hopf algebra.

\medskip Let $\w$ be a normalized 3-cocycle of $G$. Proposition \ref{p:degenerate} gives necessary and sufficient conditions for the modularity of $D^\w(G,A)$. Namely, $\res^G_A (\w)$ is a nontrivial 3-cocycle of $A$. Since $G$ is a unique involution, the Artin-Tate theory says that $\res: H^3(G, \BC^\times)_2 \to H^3(P, \BC^\times)$ is an isomorphism and $H^3(P, \BC^\times)$ is a cyclic group of order $|P|$, where $P$ is a Sylow 2-subgroup of $G$. Moreover, $\res: H^3(P, \BC^\times)_2 \to H^3(A, \BC^\times) \cong \BZ_2$ is surjective. Therefore, $\res_A^G\w$ is nontrivial if, and only if, $[\w_P]$ is a generator of $H^3(P, \BC^\times)$, i.e., $[\w]$ contains 2-generator.\ On the other hand, $\res_A^G\w$ is nontrivial if, only only if, $[\w] \not\in \ker \res_A^G$. 
Therefore, exactly one half of the classes $[\w]$ give modular quasi-Hopf algebras. 
\end{proof}

\begin{remark}
The $\w$-admissibility of $A$ implies the existence of $D^\w(G,A)$ which, in turn, depends on the choice of an $\w$-admissible pair of $A$.\ The preceding statement holds for any $\w$-admissible pair of $A$.\ However, this is not the case for super-modularity of $D^\w(G,A)$.\ The following example demonstrates this difference. 
\end{remark}

\begin{example}
Let $G = \langle z \rangle$ be a multiplicative group of order 2, $A = G$ and $\w = 1$, the constant 3-cocycle of $G$.\ Then 
$D^\w(G) = D(G)$, $\theta_z = \g_z = 1$. We simply take $\t_g = 1$ for all $g \in G$.\ Then $\b(a,b)$ is the trivial character $G$ for all $a, b \in G$.\ Therefore, $\delta \nu = \b$ for $\nu \in \Hom(G, \widehat{G})$.\ Note that for any $\nu \in\Hom(G, \widehat{G})$, the bicharacter $(\cdot, \cdot)_{\nu, \t}$ on $A$ is totally degenerate.\ Since $\FPdim \AA = \FPdim D^\w(G,A) = 2$, it follows from Proposition \ref{p:degenerate} that $\AA = \R(D^\w(G,A))$.

If $\nu(z)(z) = 1$, then $d_{\t, \nu}(z,z) = 1$ and so $D^\w(G,A) = \AA$ is braided tensor equivalent to $\R(G)$.\ However, 
if $\nu(z)(z) = {-}1$, then $d_{\t, \nu}(z,z) = {-}1$ and so $D^\w(G,A) = \AA$ is braided tensor equivalent to $\SVec$. 
\end{example}

The next Theorem demonstrates that the preceding example is almost the only exception.\ This is mainly due to the following easy Lemma about 2-groups.

\begin{lem}\label{l:2sylow}
Let $P$ be a 2-group such that $P$ has a unique involution $a$, and $|P|\geq 4$.\ Then $\chi(a) = 1$ for all characters $\chi$ of $P$ of order $2$.
\end{lem}
\begin{proof}
If $P$ is a nonabelian then $a \in P'$ and therefore $\chi(a) = 1$ for all characters $\chi$. Otherwise, $P$ is a cyclic 2-group of order greater than $2$.\ Then the unique character $\chi$ of order 2 is the square of another character, say $\xi$, of $P$.\ Then we have 
$\chi(a) = \xi(a)^2 = \xi(a^2) = 1$. 
\end{proof}

\begin{thm}\label{t:supermodular} Suppose that $G$ has a unique subgroup $A$ of order
$2$.\ Let $\w$ be a normalized $3$-cocycle on $G$.\ Then the following hold : 
\begin{enumerate}
\item[\rm (i)] If $D^{\w}(G, A)$ is a super-modular quasi-Hopf algebra then
$[\w] = [\eta^2]$ for some $[\eta]$ containing a $2$-generator of $H^3(G, \BC^{\times})$. 
\item[\rm (ii)] Conversely, suppose $[\w] = [\eta^2]$ for some $[\eta]$ containing a $2$-generator of $H^3(G, \BC^{\times})$.\ Then
\begin{enumerate}
\item[\rm (a)] $D^\w(G,A)$ is super-modular (for any choice of $\w$-admissible pair for $A$) if $A$ is not a Sylow 2-subgroup of $G$.
\item[\rm (b)] If $A$ is a Sylow 2-subgroup of $G$, there exist some $\w$-admissible pair $(\t,\nu)$ such that $D^\w(G,A)$ is super-modular. 
\end{enumerate}
\end{enumerate}
\rm (In Case (b), $G \cong A \times Q$ with $|Q|$ odd. For further discussion of these two possibilities cf.\ Subsection \ref{SSex}.)
 \end{thm}
\begin{proof}(i) Let $A:=\langle a\rangle$.\ If $D^{\w}(G, A)$ is super-modular with an $\w$-admissible pair $(\t,\nu)$ for $A$, then $d_{\t, \nu}(a,a) = {-}1$ by Proposition \ref{p:degenerate}.\ It follows from Lemma \ref{l:relation} and the 2-periodicity of $G$ that
$$
\w(a,a,a) = d_{\t, \nu}(a,a)^{-2} = 1\,.
$$
Since $G$ is 2-periodic, $[\w]$ does not contain a 2-generator of $H^3(G, \BC^\times)$.\ Therefore, there exists $[\eta] \in H^3(G, \BC^\times)$ containing a 2-generator of $H^3(G, \BC^\times)$ such that $\eta^{2n} = \w$ for some positive integer $n$.\ It suffices to show that $n$ is odd. 

\medskip
By Theorem \ref{thmsplit1}, $A$ is also $\eta$-admissible of $A$. Let $(\t', \nu')$ be an $\eta$-admissible pair of $A$. It follows from Lemma \ref{l:relation} that
$$
-1 = \eta(a,a,a) = \frac{\nu'(a)(a)^2}{\t'_a(a)^2} = \frac{1}{\t'_a(a)^2} \,.
$$
Thus, ${\t'_a(a)}^2=-1$. Set $\t'' = \{\t'_1, {\t'_a}^{2n}\}$ and $\nu'' = \nu'^{2n}$. Then $(\t'', \nu'')$ is another $\w$-admissible pair of $A$. By Lemma \ref{l:relation}, there exists group homomorphism $f: A \to \widehat{G}$ such that
$$
\frac{\t_a(x)}{\nu(a)(x)}= \frac{\t''_a(x)}{\nu''(a)(x)} \cdot f(a)(x)\ \ (x \in G).
$$
 In particular, by setting $x = a$, we have
$$
{-}1 = \frac{\t_a(a)}{\nu(a)(a)} = (-1)^n \cdot f(a)(a)\,.
$$

Let $P$ be a Sylow 2-subgroup of $G$. If $|P| = 2$, then $[\w]$ must have odd order and so the statement is clear.\ Now assume $|P| > 2$.\
 Since $G$ is 2-periodic, $P$ is either a generalized quaternion or a cyclic group.\ Since the order of $f(a)$ is at most 2, it follows from Lemma \ref{l:2sylow} that $f(a)(a) = 1$.\ Thus, the preceding equality implies $n$ must be odd. 

\medskip
(ii) Suppose $[\w] = [\eta^2]$ for some $[\eta]$ containing a $2$-generator of $H^3(G, \BC^{\times})$.\ We may assume without loss that
$\eta^2 = \w$.\ Let $(\t, \nu)$ and $(\t', \nu')$ be respectively $\w$-admissible and $\eta$-admissible pairs of $A$. Set $\t'' = \{\t'_1, {\t'_a}^2\}$ and $\nu'' = \nu'^2$.\ Then $(\t'', \nu'')$ is also an $\w$-admissible pair of $A$. By the same argument as before, we have
\begin{equation} \label{eq:super}
{-}1 = \eta(a,a,a) = \frac{1}{{\t'_a(a)}^2}\,.
\end{equation}
Again by Lemma \ref{l:relation} there exists a group homomorphism $f : A \to \widehat{G}$ such that 
$$
\frac{\t_a(a)}{\nu(a)(a)} = \frac{\t''_a(a)}{\nu''(a)(a)} \cdot f(a)(a) = \frac{{\t'_a(a)}^2}{{\nu'(a)(a)}^2} \cdot f(a)(a) = {-}f(a)(a) \,.
$$
Let $P$ be a Sylow 2-subgroup of $G$. Then $P$ is either a cyclic group or a generalized quaternion. If $A \ne P$, then $|P|\ge 4$. It follows from Lemma \ref{l:2sylow} that $f(a)(a) = 1$ and so 
$$
\frac{\t_a(a)}{\nu(a)(a)} = {-}1\,.
$$
It follows from Proposition \ref{p:degenerate} that $D^\w(G,A)$ is super-modular, and this proves (a). 

\medskip
(b) If $P = A$, it follows from \eqref{eq:super} that the $\w$-admissible pair $(\t'', \nu'')$ of $A$ satisfies
$$
 \frac{\t''_a(a)}{\nu''(a)(a)} = \frac{{\t'_a(a)}^2}{{\nu'(a)(a)}^2} = {-}1\,. 
$$
Therefore, the associated $D^\w(G,A)$ is super-modular by Proposition \ref{p:degenerate}. 
\end{proof}

\begin{remark} A modular tensor category $\CC$ is called a minimal modular extension of a super-modular category $\BB$ if $\BB$ is a braided fusion subcategory of $\CC$ such that $\FPdim(\CC) = 2 \FPdim(\BB)$ (cf. \cite{BGHN}).\ If $A\subseteq Z(G)$ is an $\w$-admissible subgroup of order $2$ and $D^\w(G,A)$ is super-modular, then $\CC=\R(D^\w(G))$ is a minimal modular extension of $\BB=\R(D^\w(G,A))$ as $\FPdim( \BB) = \dim D^\w(G,A) = |G|^2/|A| = |G|^2/2$ and $\FPdim \CC = |G|^2$.
\end{remark}


\section{Realizations of some modular tensor categories}
\subsection{Reconstruction}\label{SSintro}
An important source, conjecturally universal, of modular tensor categories are the module categories of (strongly regular) vertex operator algebras (VOAs).\ The question then arises as to whether we can find such a VOA $V$ that corresponds in this way to the MTCs $\Rep(D^{\omega}(G, A))$ constructed in our Main Theorem.\ This is the problem of \textit{reconstruction}.\ One expects that $V$ should be constructed in some way as an \textit{orbifold}, and more precisely from a pair $(U, G)$ where $U$ is another strongly regular VOA that admits $G/A$ as a group of automorphisms, and $V=U^{(G/A)}$ is the subVOA of $G/A$-fixed-points.\ Generally, one can expect the problem of reconstruction to be a difficult one, and our main intent in this Subsection is to discuss some aspects of reconstruction for the modular tensor categories associated to some of the groups with one involution listed in Subsection \ref{SSex}.\ 

\medskip
In particular, for the case of binary polyhedral groups we propose
in Conjecture \ref{con1} below a specific solution (at least for certain choices of $\omega$) and we discuss what we know about the proof.\ What is required is a comparison of the modular data coming from both $D^{\omega}(G,A)$ and the orbifolds $V$, and it is often the latter that turn out to be an obstruction.\ Recalling the ADE classification of binary polyhedral groups,
we are able to prove the Conjecture in the case of type $A$.\ For type $D$, we fall short of a complete discussion though the required orbifold modular data is in the literature. As for type $E$, the mathematical literature seems not to contain the requisite data (the same cannot be said for physics), at least for the binary octahedral and icosahedral groups.\ As for the binary tetrahedral group, we present a detailed proof based on calculations of Dong, C.\ Jiang, Q.\ Jiang, Jiao and Yu \cite{DJ}, \cite{DJJJY}.\ This may serve as a cautionary tale for readers who may want to try their hand at the other two cases.\ These examples also demonstrate how the orbifold modular data may sometimes be obtained from the modular quasi-Hopf algebra $D^\w(G,A)$. 

\medskip
In the final Subsection we present a parallel conjecture that concerns the two sporadic simple groups $J_2$ and $Co_1$, the Hall-Janko and largest Conway group respectively.\ Here we lean on the material about representation groups developed in Subsection \ref{SSrepgp}.\ These two  sporadic groups and their covering groups contain many involutions, not just one, nevertheless their
$3^{\text{rd}}$ (multiplicative) cohomology groups are determined by restriction to a certain subgroup (the \textit{categorical Schur detector} in the language of  Johnson-Freyd and Treumann \cite{JFT}) that itself \textit{does} contain a unique involution.\ This is one way in which these two cases run parallel to the binary polyhedral cases, but the analogy seems to go much deeper, all the way to the orbifold setting.

\subsection{A Conjecture for Binary polyhedral groups}\label{SSbinarypoly}
Let $L_2$ denote the $A_1$ root lattice and let $V:=V_{L_2}$ be the corresponding lattice VOA.\ It is well-known that $V$ is isomorphic to the affine algebra VOA (WZW model) determined by the Lie algebra $\mathfrak{sl}_2(\BC)$ at level $1$.\ The weight $1$ part $V_1$ of $V$ may be naturally identified with $\mathfrak{sl}_2(\BC)$.\ The automorphism group $\Aut(V)$ of $V$ can be obtained by exponentiation of this Lie algebra,
yielding the adjoint form, that is $\Aut(V)\cong PSL_2(\BC)$.\ $V$ has a unique irreducible module inequivalent to $V$, call it $W$. $\Aut(V)$ acts projectively on $W$, and its linearization defines an action of $SL_2(\BC)$.\ In this way, $SL_2(\BC)$ is the automorphism group of the intertwining algebra $V\oplus W$.

\medskip
Now consider the subgroup $SO_3(\BR)\subseteq \Aut(V)$ and its universal central extension $SU(2)$.\ Set $A:=Z(SU(2))$.\ The main focus of our interest is in the finite subgroups $G$ satisfying
$A\subseteq G \subseteq SU(2)$.\ Set $\ol{G}:= G/A$.\ Based on what we have said, it is clear that $\ol{G}$ is a group of automorphisms of $V$. And of course the groups $G$ are binary polyhedral groups as discussed in Subsection \ref{SSex}.

\medskip We have $H^4(BSU(2), \BZ)\cong \BZ$ which is isomorphic to the (group) cohomology group $H^4(SU(2), \BZ)$.\ Let $\zeta$ be a generator. Restriction of $\zeta$
to $A$ is \textit{nontrivial}.\  It follows that $\zeta_G:=\res_G \zeta$ is a generator of $H^4(G, \BZ)$ and we let $[\omega]$ be the corresponding multiplicative generator of
$H^3(G, \BC^{\times})$.\ By our Main Theorem we know that $D^\w(G, A)$ is a modular quasi-Hopf algebra.\ We can now state

\begin{q}\label{con1}
 Let the notation be as above.\ For some choice of $\zeta$ there is an equivalence of modular tensor categories $V^{\ol{G}}\mbox{-mod} \simeq \Rep(D^{\omega}(G, A))$ for some $\w$-admissible pair of $A$. 
\end{q}

This Conjecture requires at the very least that $V^{\ol{G}}$-mod is a modular tensor category, and this is known in almost all cases.\ In general, if $V^{\ol{G}}$ is \textit{strongly regular} (i.e., a  self-dual, rational, $C_2$-cofinite  vertex operator algebra of CFT type) such that conformal weights of all the nontrivial simple modules are positive, then $V$-mod is a modular tensor category \cite{Hu}.  It is well-known (see, for example, \cite{DG}) that if $\ol{G}$ is cyclic or dihedral (type $A$ or $D$) then
the corresponding orbifold $V^{\ol{G}}$ is isomorphic to either a lattice theory $V_L$ or a symmetrized lattice theory $V_L^+$ respectively
for some positive definite even lattice $L$.\ Thus if $\ol{G}$ is cyclic then
$V^{\ol{G}}\mbox{-mod} \simeq V_L\mbox{-mod}$ is a modular tensor category, because lattice VOAs are strongly regular and satisfy the needed condition on conformal weights.\ The same conclusion also holds if $\ol{G}$ is of type $D$ (dihedral) or isomorphic to $A_4$ (the binary tetrahedral case), for in these cases the needed modular data
has been verified in \cite{DN} and \cite{DJ} respectively.

\subsection{The case of type $A$}

In this Subsection we partially prove Conjecture \ref{con1} for type $A$ binary polyhedral groups. 

The order $k$ cyclic subgroups of $SO_3(\BR)$ are conjugate. Therefore, it suffices to consider the automorphism $\s$ of $V_{L_2}$ with diagonal action (cf. \cite[Section 8]{DNR}), namely
$$
\s(u \o e^{m \a}) = \ee\genfrac(){0.4pt}{1}{m}{k} u \o e^{m \a} \quad \text{ for }u \in M(1), \, m \in \BZ\,,
$$
where $\ee(r) := \exp(2\pi i r)$ for any $r \in \BQ$. It is clear that $\ord(\s)=k$ and $V^{\langle \s \rangle}_{L_2} = V_{L_{2k^2}}$. This proves the following well-known lemma.
\begin{lem} \label{l:typeA} Let $\ol G$ be a cyclic subgroup of $SO_3(\BR)$. Then 
$V^{\ol G}_{L_2}\cong V_{L_{2k^2}}$, where $k = |\ol G|$.  $\hfill\Box$
\end{lem}

\begin{thm}\label{t:typeA}
Let $G$ be an order $2k$ cyclic subgroup of $SU(2)$. There exists a generator $[\w]$ of $H^3(G,\BC^\times)$ such that $\Rep(D^\w(G,A))\simeq V_{L_2}^{\ol G}\mbox{-mod}$ as modular tensor categories, where $G/A = \ol G$. Conversely, if such equivalence holds for some 3-cocycle $\w$, then $[\w]$ is a generator of $H^3(G,\BC^\times)$. Moreover, if $K$ is a subgroup of $G$ containing $A$, then 
$\Rep(D^{\w_K}(K,A))\simeq V_{L_2}^{\ol K}\mbox{-mod}$ as modular tensor categories.
\end{thm}

\begin{proof} Since $V_{L_2}^{\ol G} \cong V_{L_{2k^2}}$, $V_{L_2}^{\ol G}\mbox{-mod}$ is a pointed modular tensor category which is determined, up to equivalence, by the quadratic form (cf. \cite[Section 8]{DNR})
\begin{equation}\label{eq:quadratic_form}
    q: L^\circ_{2k^2}/L_{2k^2} \to \BC^\times, \, q(\genfrac{}{}{0.4pt}{1}{1}{2k^2} \b+ L_{2k^2}) := \ee\genfrac(){0.4pt}{1}{1}{4k^2},
\end{equation}
where $\b$ is the generator of $L_{2k^2}$ with $(\b, \b)=2k^2$. Note that $q(\l) = \ee(w_\l)$ for any coset $\l \in L^\circ_{2k^2}/L_{2k^2}$ where $w_\l$ is the conformal weight of the $V_{L_{2k^2}}$-module corresponding to $\l$.

Now, we will show that there exists a generator $[\w]$ of $H^3(G, \BC^\times)$ such that $\Rep(D^\w(G,A))$ is pointed and determined by a quadratic form equivalent to \eqref{eq:quadratic_form}. Let $z$ be a generator $G$ and consider the generator $[\eta]$ of $H^3(G, \BC^\times)$  given by
\begin{equation}\label{eq:gen}
    \eta(z^a, z^b, z^c) = \ee\genfrac(){0.4pt}{0}{ \ol a(\ol b + \ol c - \ol{b+c})}{4k^2}
\end{equation}
for $a,b, c \in \BZ$, where $\ol x$ denotes the remainder upon the division of $x$ by $2k$. Let $\w = \eta^{m^2}$ where $m$ is an integer coprime to $2k$. It is immediate to see that  $\w(z^a, z^b, z^c)= \w(z^a, z^c, z^b)$. Therefore, 
$$
\theta_{z^a}(z^b, z^c) = \w(z^a, z^b, z^c) =\ee\left(\frac{ m^2\ol a(\ol b + \ol c - \ol{b+c})}{4k^2}\right)\,.
$$
Define $\tau_{z^a}(z^b) =\ee\left(\frac{ m^2\ol{a} \ol{b}}{4k^2}\right)$ for any $a,b \in \BZ$. Then $\delta \tau_{z^a} = \theta_{z^a}$ for $a \in \BZ$. Direct computation shows
$$
\b_\tau(z^k,z^k)(z^a) = \ee\left(\frac{m^2\ol a}{k}\right)\,.
$$
Thus, if we set $\nu(z^k)(z^a) =  \ee(\frac{m^2\ol a}{2k})$, then $(\tau, \nu)$ is an  $\w$-admissible pair of $A$. 

 Since $D^\w(G)$ is commutative, $\Rep(D^\w(G))$ is a pointed modular tensor category. The quadratic of form $(\G^\w_0(G),q_\w)$ determines the equivalence class of $\Rep(D^\w(G))$, where
$$
\G^\w_0(G) = \{ u(\chi,z^a)=\sum_{b \in \BZ_{2k}} \chi(z^b)\tau_{z^a}(z^b) e_{z^b} \, z^a \,\mid \chi \in \widehat{G}, a \in \BZ \}
$$
is the group of central group-like elements of $D^\w(G)$ and
$$
q_\w(u(\chi,z^a)) = \chi(z^a) \tau_{z^a}(z^a)\,.
$$
The associated bicharacter $\bb_\w$ on $\G^\w_0(G)$ is given by 
$$
\bb_\w(u(\chi_1 ,z^{a_1}), u(\chi_2,z^{a_2})) = \chi_2(z^{a_1}) \chi_1(z^{a_2}) \tau_{z^{a_2}}(z^{a_1})\tau_{z^{a_1}}(z^{a_2})\,.
$$

The element $u(\nu^{-1}(z^k), z^k) \in \G^\w_0(G)$ is of order 2 and 
$$
q_\w(u(\nu^{-1}(z^k), z^k))= \frac{\tau_{z^k}(z^k)}{\nu(z^k)(z^k)}=-i.
$$
Moreover,  
$$
\bb_\w(u(1,z), u(\nu^{-1}(z^k),z^k))= \tau_{z^k}(z)\tau_z(z^k)/\nu(z^k)(z) =1 . 
$$
Therefore, $\langle u(1, z) \rangle$ is the orthogonal complement of $\langle u(\nu^{-1}(z^k), z^k) \rangle$ in $\G^\w_0(G)$.  Since $\bb_\w(u(1,z), u(1,z)) =\ee\left(\frac{m^2}{2k^2}\right)$, $2k^2 \mid \ord(u(1,z))$. Therefore, the quadratic form determined by $\Rep(D^\w(G,A))$ is $(\langle u(1, z) \rangle, q_\w)$. 

Let $\tilde m$ be the inverse of $m$ modulo $2k$. Since $\langle u(1, z) \rangle = \langle u(1,z)^{\tilde m} \rangle$ and 
$$
q_\w(u(1,z)^{\tilde m}) = \tau_{z^{\tilde m}}({z^{\tilde m}}) = \ee\left(\frac{m^2 \tilde m^2}{4k^2}\right) = \ee\left(\frac{1}{4k^2}\right),
$$
the quadratic form $(\langle u(1, z) \rangle, q_\w)$ is equivalent to the one shown in \eqref{eq:quadratic_form}. Therefore, $V_{L_{2k^2}}$-mod is equivalent to $\Rep(D^\w(G,A))$ as modular tensor categories. 

If $K$ is a subgroup of $G$ containing $A$, then $K = \langle z^\ell \rangle$ for some positive integer $\ell$ such that $2k/\ell$ is even. It is immediate to see that $\eta_K$ has the same formula \eqref{eq:gen} for the cyclic group $K$ of order $2k/\ell$. Thus, $\w_K = \eta_K^m$ and so $\Rep(D^{\w_K}(K,A))$  is equivalent to $L_2^{\ol K}$-mod as modular tensor categories by the same proof for $G$.

If $\Rep(D^\w(G,A)) \simeq V_{L_2}^{\ol G}\mbox{-mod}$ for some $[\w] \in H^3(G, \BC^\times)$, then the Frobenius-Schur exponent (or the order of the $T$-matrix) of $V_{L_2}^{\ol G}\mbox{-mod}$ is $4k^2$. By \cite[Thm. 9.3]{NS}, the Frobenius-Schur exponent $\FSexp(D^\w(G))$ of $\Rep(D^\w(G))$ is given by 
$$
\FSexp(D^\w(G)) = \ord([\w])\cdot |G| =2k \cdot \ord([\w]) .
$$ 
Since $\Rep(D^\w(G,A))$ is a modular subcategory of $\Rep(D^\w(G))$, $\FSexp(D^\w(G,A))$ divides $\FSexp(D^\w(G))$, and hence $2k \mid \ord([\w])$. Therefore, $[\w]$ is a generator of $H^3(G, \BC^\times)$.  
\end{proof}
\begin{remark}
There are two $\w$-admissible pairs for $A$. The other one is given by $(\tau, f^{-1}\nu)$ where $f \in \Hom(A, \widehat{G})$ is nontrivial. In this case, $u(f \nu^{-1}(z^k),z^k)$ is another order 2 element of $\G^\w_0(G)$ and $q_\w(u(f \nu^{-1}(z^k),z^k)) = \frac{f(z^k)(z^k) \tau_{z^k}(z^k)}{\nu(z^k)(z^k)} = (-1)^{k+1}i$. If $k$ is odd, then $u(f(z^k),z)$ is orthogonal to $u(f \nu^{-1}(z^k),z^k)$ and 
$$
q_\w(u(f(z^k),z))=-\ee\genfrac(){0.4pt}{1}{1}{4k^2}= \ee\genfrac(){0.4pt}{1}{1+2k^2}{4k^2}.
$$
Therefore, this $\w$-admissible pair $(\tau, f^{-1}\nu)$ yields the $\Rep(D^\w(G,A))$ which is  inequivalent to the one in the preceding theorem as $1+2k^2$ is not a square modulo $4k^2$.

On the other hand, if $k$ is even,  then $u(\chi,z)$ is orthogonal to $u(f \nu^{-1}(z^k),z^k)$ where $\chi$ is a generator of $\widehat{G}$. In this case, 
$q_\w(u(\chi,z))=\ee\left(\frac{1+2k a}{4k^2}\right)$ for some integer $a$ coprime to $2k$.  If $1+2ka$ is not a square modulo $4k^2$, then the associated $\Rep(D^\w(G,A))$ is not equivalent to the one in the preceding theorem.
\end{remark}
\subsection{The modular data of $V^{A_4}$ and $V^{D_4}$}
We have already pointed out that  $V^{A_4}$ is shown in \cite{DJ} to be a strongly regular VOA with positive conformal weights for all nontrivial irreducible $V^{A_4}$-modules.\ And indeed,  $V^{D_4} \cong V_{L_8}^+$ where $L_8$ is a rank 1 lattice with the generator $\beta$ satisfying $(\b, \b)=8$. The fusion rules, conformal weights and part of the $S$-matrix of $V^{A_4}$-mod were also computed in \cite{DJ, DJJJY}. 

The VOA $V_{L_2}^{A_4}$ has 21 irreducible modules which are denoted by $M_0, \cdots, M_{20}$ where $M_0$ is the trivial module. These objects are labelled in the same order as in \cite{DJJJY}. Their conformal weights $w_i$ are given by
$$
\arraycolsep=1.9pt\def\arraystretch{1.3}
\begin{array}{c|ccccccccccccccccccccc}
j& 0 &  1 &  2 &  3 &  4 &  5 &  6 &  7 &  8 &  9 & {10} & {11} & {12} & {13} 
& {14} & {15} & {16} & {17} & {18} & {19} & {20}\\ \hline
w_j & 0 &4 & 4 & 1 & \frac{1}{16} & \frac{9}{16} & \frac{1}{36} & \frac{25}{36} & \frac{49}{36}& 
 \frac{1}{9} & \frac{4}{9} & \frac{16}{9} & \frac{1}{36} & \frac{25}{36} &\frac{49}{36}& 
 \frac{1}{9} & \frac{4}{9} & \frac{16}{9} & \frac{1}{4} & \frac{9}{4} & \frac{9}{4} 
\end{array}\,.
$$
Therefore, the twist $\theta_j = e^{2 \pi i w_j}$ of $M_j$ are given by
$$
\arraycolsep=1.9pt\def\arraystretch{1.3}
\begin{array}{c|ccccccccccccccccccccc}
j & 0 &  1 &  2 &  3 &  4 &  5 &  6 &  7 &  8 &  9 & {10} & {11} & {12} & {13} 
& {14} & {15} & {16} & {17} & {18} & {19} & {20}\\ \hline
\theta_j & 1 &1 & 1 & 1 & \zeta_{16} & \zeta_{16}^9 & \zeta_{36} & \zeta_{36}^{25} & \zeta_{36}^{13}& 
 \zeta_{9} & \zeta_{9}^4 & \zeta_{9}^7 & \zeta_{36} & \zeta_{36}^{25} & \zeta_{36}^{13}& 
 \zeta_{9} & \zeta_{9}^4 & \zeta_{9}^7 & i & i & i 
\end{array}\,.
$$
where $\zeta_n=\ee\genfrac(){0.4pt}{1}{1}{n}$. In this case, $G \cong SL_2(3)$. Since $\widehat{G} \cong \BZ_3$,  $\Hom(A, \widehat{G})$ is trivial. By Lemma \ref{l:relation}, there is only one inequivalent $\w$-admissible pair of $A$ for each $[\w] \in H^3(G, \BC^\times)$.  We computed the modular data of  all possible $D^\w(G, A)$  by GAP.  and we found exactly one cohomology class $[\w_0]$ such that the (unnormalized) $T$-matrix $\tilde T$ of $\Rep(D^{\w_0}(G,A))$, after reordering the simple objects, coincides with the twists of $V^{\ol{G}}$, i.e.,
$$
\tilde T = \diag( 1 ,\,1 ,\, 1 ,\, 1 ,\, \zeta_{16} ,\, \zeta_{16}^9 ,\, \zeta_{36} ,\, \zeta_{36}^{25} ,\, \zeta_{36}^{13},\, \zeta_{9} ,\, \zeta_{9}^4 ,\, \zeta_{9}^7 ,\, \zeta_{36} ,\, \zeta_{36}^{25} ,\, \zeta_{36}^{13},\, 
 \zeta_{9} ,\, \zeta_{9}^4 ,\, \zeta_{9}^7 ,\, i ,\, i ,\, i)\,.
$$

The corresponding unnormalized $S$-matrix $\tilde S$ of  $\Rep(D^{\w_0}(G,A))$ is given by
{\small
$$
\arraycolsep=1.4pt\def\arraystretch{1.1}
\begin{array}{c|cccccccccccccccccccccc}
\frac{\tilde S_{ij}}{4} & 0 & 1 & 2 & 3 & 4 & 5 & 6 & 7 & 8 & 9 & 10 & 11 & 12 &13 &14 & 15 & 16 & 17 & 18 & 19 & 20\\ \hline
0 &\frac{1}{4} & \frac{1}{4} & \frac{1}{4} & \frac{3}{4} & \frac{3}{2} & \frac{3}{2} & 1 & 1 & 1 & 1 & 1 & 1    & 1 & 1 & 1 & 1 & 1 & 1 & \frac{1}{2} & \frac{1}{2} & \frac{1}{2} \\
1 &\frac{1}{4} & \frac{1}{4} & \frac{1}{4} & \frac{3}{4} & \frac{3}{2} & \frac{3}{2} & \ol\zeta_3 & \ol\zeta_3  & \ol\zeta_3 & \ol\zeta_3 & \ol\zeta_3 & \ol\zeta_3 & \zeta_3 & \zeta_3 & \zeta_3 & \zeta_3 & \zeta_3 & \zeta_3 & \frac{1}{2} & \frac{1}{2} & \frac{1}{2} \\
2 &\frac{1}{4} & \frac{1}{4} & \frac{1}{4} & \frac{3}{4} & \frac{3}{2} & \frac{3}{2} & \zeta_3 & \zeta_3 & \zeta_3 & \zeta_3 & \zeta_3 & \zeta_3 & \ol\zeta_3 & \ol\zeta_3 & \ol\zeta_3 & \ol\zeta_3 & \ol\zeta_3 & \ol\zeta_3 & \frac{1}{2} & \frac{1}{2} & \frac{1}{2} \\
3 &\frac{3}{4} & \frac{3}{4} & \frac{3}{4} & \frac{9}{4} & \frac{-3}{2} & \frac{-3}{2} & 0 & 0 & 0 & 0 & 0 & 0 & 0 & 0 & 0 & 0 & 0 & 0 & \frac{3}{2} & \frac{3}{2} & \frac{3}{2} \\
4 &\frac{3}{2} & \frac{3}{2} & \frac{3}{2} & \frac{-3}{2} & \frac{3}{\sqrt{2}} & \frac{-3}{\sqrt{2}} & 0 & 0 & 0 & 0 & 0 & 0 & 0 & 0 & 0 & 0 & 0 & 0 & 0 & 0 & 0 \\
5 &\frac{3}{2} & \frac{3}{2} & \frac{3}{2} & \frac{-3}{2} & \frac{-3}{\sqrt{2}} & \frac{3}{\sqrt{2}} & 0 & 0 & 0 & 0 & 0 & 0 & 0 & 0 & 0 & 0 & 0 & 0 & 0 & 0 & 0 \\
6 & 1 & \ol\zeta_3 & \zeta_3 & 0 & 0 & 0 & \ol\zeta_{18} & \zeta_{18}^5 & \zeta_{18}^{11} & \zeta_9 & \zeta_9^7 & \zeta_9^4 & \zeta_{18} &  \zeta_{18}^{13} & \zeta_{18}^7 & \ol\zeta_9 &\zeta_9^2 & \zeta_9^5 & 1 & \ol\zeta_3 & \zeta_3 \\
7 & 1 & \ol\zeta_3 & \zeta_3 & 0 & 0 & 0 & \zeta_{18}^5 & \zeta_{18}^{11} & \ol\zeta_{18} & \zeta_9^4 & \zeta_9 & \zeta_9^7 & \zeta_{18}^{13} & \zeta_{18}^7 & \zeta_{18} & \zeta_9^5 & \ol\zeta_9 & \zeta_9^2 & 1 & \ol\zeta_3 & \zeta_3 \\
8 & 1 & \ol\zeta_3 & \zeta_3 & 0 & 0 & 0 & \zeta_{18}^{11} & \ol\zeta_{18} & \zeta_{18}^5 & \zeta_9^7 & \zeta_9^4 & \zeta_9 & \zeta_{18}^7 & \zeta_{18} & \zeta_{18}^{13} & \zeta_9^2 & \zeta_9^5 & \ol\zeta_9 & 1 & \ol\zeta_3 & \zeta_3 \\
9 & 1 & \ol\zeta_3 & \zeta_3 & 0 & 0 & 0 & \zeta_9 & \zeta_9^4 & \zeta_9^7 & \zeta_9^7 & \zeta_9^4 & \zeta_9 & \ol\zeta_9 & \zeta_9^5 & \zeta_9^2 & \zeta_9^2 & \zeta_9^5 & \ol\zeta_9 & -1 & \zeta_6 & \ol\zeta_6 \\
10 & 1 & \ol\zeta_3 & \zeta_3 & 0 & 0 & 0 & \zeta_9^7 & \zeta_9 & \zeta_9^4 & \zeta_9^4 & \zeta_9 & \zeta_9^7 & \zeta_9^2 & \ol\zeta_9 & \zeta_9^5 & \zeta_9^5 & \ol\zeta_9 & \zeta_9^2 & -1 & \zeta_6 & \ol\zeta_6 \\
11 & 1 & \ol\zeta_3 & \zeta_3 & 0 & 0 & 0 & \zeta_9^4 & \zeta_9^7 & \zeta_9 & \zeta_9 & \zeta_9^7 & \zeta_9^4 & \zeta_9^5 & \zeta_9^2 & \ol\zeta_9 & \ol\zeta_9 & \zeta_9^2 & \zeta_9^5 & -1 & \zeta_6 & \ol\zeta_6 \\
12 & 1 & \zeta_3 & \ol\zeta_3 & 0 & 0 & 0 & \zeta_{18} & \zeta_{18}^{13} & \zeta_{18}^7 & \ol\zeta_9 & \zeta_9^2 & \zeta_9^5 & \ol\zeta_{18} & \zeta_{18}^5 & \zeta_{18}^{11} & \zeta_9 & \zeta_9^7 & \zeta_9^4 & 1 & \zeta_3 & \ol\zeta_3 \\
13 & 1 & \zeta_3 & \ol\zeta_3 & 0 & 0 & 0 & \zeta_{18}^{13} & \zeta_{18}^7 & \zeta_{18} & \zeta_9^5 & \ol\zeta_9 & \zeta_9^2 & \zeta_{18}^5 & \zeta_{18}^{11} & \ol\zeta_{18} &  \zeta_9^4 & \zeta_9 & \zeta_9^7 & 1 & \zeta_3 & \ol\zeta_3 \\
14 & 1 & \zeta_3 & \ol\zeta_3 & 0 & 0 & 0 & \zeta_{18}^7 & \zeta_{18} & \zeta_{18}^{13} & \zeta_9^2 & \zeta_9^5 & \ol\zeta_9 & \zeta_{18}^{11} & \ol\zeta_{18} & \zeta_{18}^5 & \zeta_9^7 & \zeta_9^4 & \zeta_9 & 1 & \zeta_3 & \ol\zeta_3 \\
15 & 1 & \zeta_3 & \ol\zeta_3 & 0 & 0 & 0 & \ol\zeta_9 & \zeta_9^5 & \zeta_9^2 & \zeta_9^2 & \zeta_9^5 & \ol\zeta_9 & \zeta_9 & \zeta_9^4 & \zeta_9^7 & \zeta_9^7 & \zeta_9^4 & \zeta_9 & -1 & \ol\zeta_6 & \zeta_6 \\
16 & 1 & \zeta_3 & \ol\zeta_3 & 0 & 0 & 0 & \zeta_9^2 & \ol\zeta_9 & \zeta_9^5 & \zeta_9^5 & \ol\zeta_9 & \zeta_9^2 & \zeta_9^7 & \zeta_9 & \zeta_9^4 & \zeta_9^4 & \zeta_9 & \zeta_9^7 & -1 & \ol\zeta_6 & \zeta_6 \\
17 & 1 & \zeta_3 & \ol\zeta_3 & 0 & 0 & 0 & \zeta_9^5 & \zeta_9^2 & \ol\zeta_9 & \ol\zeta_9 & \zeta_9^2 & \zeta_9^5 & \zeta_9^4 & \zeta_9^7 & \zeta_9 & \zeta_9 & \zeta_9^7 & \zeta_9^4 & -1 & \ol\zeta_6 & \zeta_6 \\
18 & \frac{1}{2} & \frac{1}{2} & \frac{1}{2} & \frac{3}{2} & 0 & 0 & 1 & 1 & 1 & -1 & -1 & -1 & 1 & 1 & 1 & -1 & -1 & -1 & -1 & -1 & -1 \\
19 & \frac{1}{2} & \frac{1}{2} & \frac{1}{2} & \frac{3}{2} & 0 & 0 & \ol\zeta_3 & \ol\zeta_3 & \ol\zeta_3 & \zeta_6 & \zeta_6 & \zeta_6 & \zeta_3 & \zeta_3 & \zeta_3 & \ol\zeta_6 & \ol\zeta_6 & \ol\zeta_6 & -1 & -1 & -1 \\
20 & \frac{1}{2} & \frac{1}{2} & \frac{1}{2} & \frac{3}{2} & 0 & 0 & \zeta_3 & \zeta_3 & \zeta_3 & \ol\zeta_6 & \ol\zeta_6 & \ol\zeta_6 & \ol\zeta_3 & \ol\zeta_3 & \ol\zeta_3 & \zeta_6 & \zeta_6 & \zeta_6 & -1 & -1 & -1
\end{array}
$$
}

\noindent
where the first row lists the labels of columns, and the first column indicates the labels of  rows. Note that $\tilde S$ is a symmetric matrix.

Column 0 and columns 7 to 17 of the $S$-matrix were also computed in \cite[Appendix A]{DJJJY}. The rest of the unnormalized $S$-matrix can be computed by the fusion rules obtained in  \cite{DJJJY}, and the formula (cf. \cite{BK}):
\begin{equation} \label{eq:S}
 \tilde S_{ij} = \sum_k N^k_{ij}\, d_k\, \ee\left(w_j +w_i-w_k\right)\,,
\end{equation}
where the fusion coefficient $N_{ij}^k$ is the dimension of the space of intertwining operators of type $\genfrac(){0pt}{1}{M_k}{M_iM_j}$. However, the partial $S$-matrix presented in  \cite[Appendix A]{DJJJY} is different from that of $D^{\w_0}(G,A)$ at the two entries $\tilde S_{10,9}$ and $\tilde S_{16,15}$. We believe the $S$-matrix displayed above is the correct one since the $2 \times 2$ blocks for the pairs $(M_9, M_{10})$ and $(M_{15}, M_{16})$ in \cite[Appendix A]{DJJJY} are not symmetric. 

To further demonstrate the validity of Conjecture \ref{con1}, we consider the orbifold $V^{D_4}$ which is isomorphic to $V_{L_8}^+$ (cf. \cite{DJ}), where the action of $D_4$ on $V$ is inherited from $A_4$. The conformal weights of  the simple $V_{L_8}^+$-modules have been computed in \cite{DN}. Let $M_0, \dots, M_{10}$ denote respectively the simple $V_{L_8}^+$-modules
$$
V_{L_8}^+,\, V_{L_8}^-,\, V_{L_8+\b/2},\, V_{L_8+\frac{\b}{8}}, \,  V_{L_8+\frac{3\b}{8}},\, V_{L_8}^{T_1, +}, \,
V_{L_8}^{T_1, -}, \, V_{L_8}^{T_2, +}, \, V_{L_8}^{T_2, +}\,.
$$
Their conformal weight $w_j$ of $M_j$ are given by
$$
\arraycolsep=2pt\def\arraystretch{1.3}
\begin{array}{c|cccccccccccc}
j  & 0 &  1 &  2 &  3 &  4 &  5 &  6 &  7 &  8 &  9 & {10} \\ \hline
w_j & 0 &1 & 1 & 1 & \frac{1}{4} & \frac{1}{16} & \frac{9}{16} & \frac{1}{16} & \frac{9}{16} & 
 \frac{1}{16} & \frac{9}{16} 
\end{array}\,.
$$
Thus, the unnormalied $T$-matrix of $V^{D_4}$-mod is
$$
\tilde T=\diag(1,1,1,1,i, \zeta_{16}, -\zeta_{16}, \zeta_{16}, -\zeta_{16},\zeta_{16}, -\zeta_{16} )\,.
$$
The fusion coefficients $N_{ij}^k$ for $V^{D_4}$-mod were computed in \cite{A}. Since $V^{D_4}$ is of CFT type and all the nontrivial simple module have positive conformal weights, the quantum dimension $d_j$ of $M_j$ is positive and hence $d_j$ is the Frobenius-Perron dimension of $M_j$, the largest real eigenvalue of the fusion matrix $N_i$, where  $(N_i)_{jk}=N_{ij}^k$. By direct computation, we find 
$$
\arraycolsep=2pt\def\arraystretch{1.3}
\begin{array}{c|cccccccccccc}
j  & 0 &  1 &  2 &  3 &  4 &  5 &  6 &  7 &  8 &  9 & {10} \\ \hline
d_j & 1 &1 & 1 & 1 & 2 & 2 & 2 & 2 & 2 & 
 2 & 2 
\end{array}\,.
$$
Using the formula \eqref{eq:S}, we find the unnormalized
$S$-matrix of $V^{D_4}$-mod:
$$
{\textstyle
\tilde S =\left[
\arraycolsep=2pt\def\arraystretch{1}
\begin{array}{rrrrrrrrrrr}
1 & 1 & 1 & 1 & 2 & 2 & 2 & 2 & 2 & 2 & 2 \\
1 & 1 & 1 & 1 & 2 & 2 & 2 & -2 & -2 & -2 & -2 \\
1 & 1 & 1 & 1 & 2 & -2 & -2 & 2 & 2 & -2 & -2 \\
1 & 1 & 1 & 1 & 2 & -2 & -2 & -2 & -2 & 2 & 2 \\
2 & 2 & 2 & 2 & -4 & 0 & 0 & 0 & 0 & 0 & 0 \\
2 & 2 & -2 & -2 & 0 & 2\sqrt{2} & -2\sqrt{2} & 0 & 0 & 0 & 0 \\
2 & 2 & -2 & -2 & 0 & -2\sqrt{2} & 2\sqrt{2} & 0 & 0 & 0 & 0 \\
2 & -2 & 2 & -2 & 0 & 0 & 0 & 2\sqrt{2} & -2\sqrt{2} & 0 & 0 \\
2 & -2 & 2 & -2 & 0 & 0 & 0 & -2\sqrt{2} & 2\sqrt{2} & 0 & 0 \\
2 & -2 & -2 & 2 & 0 & 0 & 0 & 0 & 0 & 2\sqrt{2} & -2\sqrt{2} \\
2 & -2 & -2 & 2 & 0 & 0 & 0 & 0 & 0 & -2\sqrt{2} & 2\sqrt{2}
\end{array}\right]\,.
}
$$

We restrict the 3-cocycle $\w_0$, and the $\w_0$-admissible pair $(\tau, 1)$ of $A$ obtained for the $A_4$-orbifold to the quaternion subgroup $Q_8$. Using the induced $(\w_0)_{Q_8}$-admissible pair $(\t|_{Q_8}, 1)$ of $A$, we find $\Rep(D^{\w_0}(Q_8, A))$ is modular, and its modular data of   coincides with the one that  we obtained above for $V^{D_4}$-mod after reordering of the simple modules.

It is worth noting that $(\w_0)_{Q_8}$ has 4 inequivalent admissible pairs parametrized by $\Hom(A, \widehat{Q}_8)$. However, only the one $(\tau|_{Q_8}, 1)$ inherited from the $A_4$ orbifold  yields the same unnormalized $T$-matrix.

Finally, we consider the orbifolds $V_{L_2}^{\ol K}$ where $\ol K$ is a cyclic subgroup of $A_4$. The module categories of the orbifolds $V_{L_2}^{\ol K}$ for all cyclic subgroups $\ol K$ of order $3$ are equivalent modular tensor categories determined by the quadratic form $(\BZ_{18}, q)$ where $q(1)=\ee\genfrac(){0.4pt}{1}{1}{36}$ (cf. Lemma \ref{l:typeA}). Similarly, the three order 2 subgroups of $A_4$ yield equivalent modular tensor categories given by the quadratic form $(\BZ_8, q')$ where $q'(1)=\ee\genfrac(){0.4pt}{1}{1}{16}$. All these modular tensor categories $V^{\ol K}_{L_2}$-mod, where $A \subset K$ are cyclic subgroups of $SL_2(3)$, are realized by $D^{\w_0}(K,A)$ with  the induced $(\w_0)_K$-admissible pair of $A$.

In conclusion, for any subgroup $K$ of $G=SL_2(3)$ containing $A$, $V_{L_2}^{\ol K}\Mod \simeq \Rep(D^{(\w_0)_K}(K, A))$ as modular tensor categories, where $\ol K = K/A$.

\subsection{Two sporadic examples involving $2.J_2$ and $Co_0$}
In this final Subsection we present another conjecture that is in many ways analogous to Conjecture 1, however in place of a binary dihedral group
we consider the perfect group $2.J_2$.\ This is the representation group of the sporadic simple group $J_2$, sometimes called the Hall-Janko group.\
In order to describe Conjecture 2 we need to explain some background.\ In this context we also consider the largest simple Conway group $Co_1$ whose representation group is $Co_0$, the automorphism group of the Leech lattice.\ For some background on these sporadic groups see, for example, \cite{At}.\ Note, in particular, that $J_2$ is a subgroup of $Co_1$ and this lifts to a containment
$2.J_2\subseteq Co_0$.\ In fact there are the following inclusions of groups as follows:
 
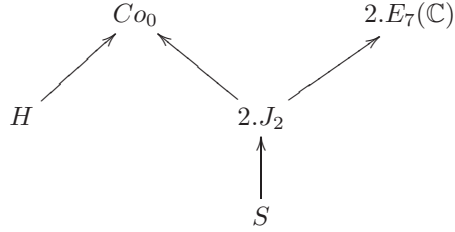
\begin{eqnarray*}
 \xymatrix{& Co_0 && 2.E_7(\BC) \\
 H\ar[ur] & & 2.J_2\ar[ul]\ar[ur]\\
 & & S\ar[u] \\
 }
 \end{eqnarray*}
The notation is as follows:\ $S\cong SL_2(5)$ and $H\cong \BZ_3\times Q_{16}$ are groups with one involution and $2.E_7(\BC)$ is the universal cover of
the Lie group $E_7(\BC)$.\ The containment of $2.J_2$ in this universal cover is proved in \cite{GR}.\ Upon taking third group cohomology, there is a corresponding diagram where all maps arise from restriction of cohomology and those with finite domain are injections:

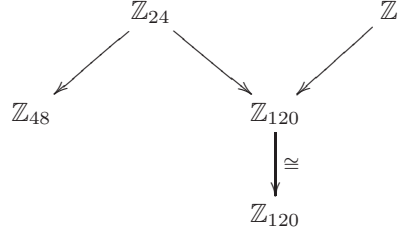
\begin{eqnarray*}
 \xymatrix{& \BZ_{24}\ar[dl]\ar[dr] && \BZ\ar[dl] \\
 \BZ_{48} & & \BZ_{120}\ar[d]^ \cong \\
 & & \BZ_{120} \\
 }
 \end{eqnarray*}
The diagram of cohomology groups, at least for the finite groups, is established in \cite{JFT}.\ The next result does not involve $E_7(\BC)$.

\begin{thm}\label{thmspor}
Let $E$ be one of the covering groups $2.J_2$ or $Co_0$, set $A=Z(E)$, and let $\w$ be a normalized 3-cocycle of $E$.\ 
Then $D^\w(E,A)$ is a quasi-Hopf algebra and the following hold:
 \begin{enumerate}
 \item[\rm (i)] If $E=2.J_2$ then $\Rep(D^{\omega}(E, A))$ is modular if, and only if, $[\omega]$ contains a $2$-generator.
 \item[\rm (ii)] If $E=Co_0$ then $\Rep(D^{\omega}(E, A))$ is \emph{not} modular; it is super-modular if, and only if, $[\w]$ contains a $2$-generator.
 \end{enumerate}
\end{thm}
\begin{proof}
 $D^{\omega}(E, A)$ is a quasi-Hopf algebra by Theorem \ref{thmE}.\ 
 
 \medskip
 \noindent
Case (i). $E=2.J_2$.\ By Proposition \ref{p:nondegenrate}, $D^{\omega}(E, A)$
is modular if, and only if, $\omega_A$ is \textit{not} a coboundary.\ Thanks to the isomorphism in the previous display, this holds just when $[\omega]$ contains a $2$-generator.\ This completes the proof of (i). \medskip\\
\noindent
Case (ii). $E=Co_0$.\ For any normalized $3$-cocycle $\omega$
on $E$, $H$ has a normalized 3-cocycle $\eta$ such that $\eta^2 = \w_H$. This implies that
$$\w_A = \res_A^H (\w_H) = \res_A^H(\eta^2) = \eta_A^2 
$$ 
is a coboundary of $A$. By Proposition \ref{p:nondegenrate}, $D^\w(E,A)$ is \emph{never} modular. 

 \medskip
By Lemma \ref{l:induced}, the super-modularity of $D^\w(E,A)$ is equivalent to the super-modularity of $D^{\w_H}(H,A)$.\ 
Since $A$ is not a Sylow 2-subgroup of $H$, it follows from Theorem \ref{t:supermodular}(ii)(a) that $D^{\w_H}(H,A)$ is super-modular if, and only if, $[\eta]$ contains a 2-generator of $H^3(H, \BC^\times)$.\ This is equivalent to $16\mid\ord([\eta])$ or $8 \mid \ord([\w_H])$.\ But we have $8 \mid \ord([\w_H])$ if, and only if, $\ord([\w]) = 8$ or $24$.\
This completes the proof of (ii).
\end{proof}

We will discuss the Reconstruction problem for the modular tensor category described in part (i) of the Theorem.

\medskip
Let $V=V_{E_7}$ be the lattice VOA defined by the $E_7$ root lattice.\ This is also the affine algebra VOA (WZW model) of type
$E_7$ and level $1$, and just like $V_{L_2}$ before, $V$ has just two simple modules $V$ and $W$.\ (For further background on this and other such VOAs, see
\cite{MNS}.) \ There is an action of the universal cover $2.E_7(\BC)$ on $V\oplus W$.\
 Because $2.J_2\subseteq 2.E_7(\BC)$ then,  just as in the binary polyhedral case, $J_2$ acts on $V_{E_7}$. 
 A generator of $H^4(B(2.E_7(\BC)), \BZ)$, call it $\zeta$, restricts to a generator of $H^4(2.J_2, \BZ)$. Let $[\omega]$ be the corresponding class in $H^3(2.J_2, \BC^{\times})$. Now, we can state

\medskip\noindent
Conjecture 2.\ For some choice of $\zeta$, there is an equivalence of modular tensor categories 
$\Rep(D^{\omega_G}(G, A))\simeq V_{E_7}^{G/A}$-mod for any subgroup $G$ of $2.J_2$ containing $A$.

\begin{remark}
The pointed modular tensor categories $V_{E_7}\Mod$ and $V_{L_2}\Mod$ are inequivalent  because the conformal weights of their nontrivial simple modules are respectively $\frac{3}{4}$ and $\frac{1}{4}$. One might therefore expect that the modular tensor categories $V_{E_7}^{G/A}\Mod$ and $V_{L_2}^{G/A}\Mod$ are inequivalent if $G/A$ is a common subgroup of $J_2$ and $SO_3(\BR)$.
\end{remark}

 \begin{remark} The reconstruction problem  for the groups $G=SL_2(q)$ discussed in Section \ref{SSex}, Case 4 and their associated modular tensor categories $\Rep(D^{\omega}(G, A))$, seems to be very challenging.\ It is not even clear if,  for a general prime power $q$ (say, $q\geq 11$), we can find suitable VOAs that admit $\overline{G}$ as a group of automorphisms.
 \end{remark}
\bigskip
\noindent
\textbf{Acknowledgement:} We thank Chongying Dong for helpful suggestions concerning references and the Lemma in Section 5.

\end{document}